\documentclass[11pt]{article}
\usepackage{amssymb,amsmath,comment}
\catcode`\@=11 \@addtoreset{equation}{section}
\def\thesection{\arabic{section}}

\def\theequation{\thesection.\arabic{equation}}
\catcode`\@=12
\usepackage{colortbl}%
\usepackage{a4wide}
\usepackage{parskip}
\newcommand{\tl}{\tilde}

\newcommand{\ds} {\displaystyle}
\newcommand{\e}{\epsilon}
\newcommand{\vep}{\varepsilon}
\newcommand{\pa} {\partial}
\newcommand{\al} {\alpha}
\newcommand{\ba} {\beta}
\newcommand{\de} {\delta}
\newcommand{\ga} {\gamma}
\newcommand{\Ga} {\Gamma}
\newcommand{\Om} {\Omega}
\newcommand{\sg}{\sigma}

\newcommand{\ov}{\overline}
\newcommand{\De} {\Delta}
\newcommand{\la} {\lambda}

\newcommand{\ka}{\kappa}
\newcommand{\noi} {\noindent}
\newcommand{\na} {\nabla}

\newcommand{\mb} {\mathbb}
\newcommand{\mc} {\mathcal}

\usepackage[all]{xy}
\catcode`\@=11
\def\theequation{\@arabic{\c@section}.\@arabic{\c@equation}}
\catcode`\@=12

\def\QED{\hfill {$\square$}\goodbreak \medskip}

\newtheorem{Theorem}{Theorem}[section]
\newtheorem{Lemma}[Theorem]{Lemma}
\newtheorem{Proposition}[Theorem]{Proposition}
\newtheorem{Corollary}[Theorem]{Corollary}
\newtheorem{Remark}{Remark}
\newtheorem{Definition}{Definition}

\def\Xint#1{\mathchoice
	{\XXint\displaystyle\textstyle{#1}}%
	{\XXint\textstyle\scriptstyle{#1}}%
	{\XXint\scriptstyle\scriptscriptstyle{#1}}%
	{\XXint\scriptscriptstyle\scriptscriptstyle{#1}}%
	\!\int}
\def\XXint#1#2#3{{\setbox0=\hbox{$#1{#2#3}{\int}$ }
		\vcenter{\hbox{$#2#3$ }}\kern-.6\wd0}}

\begin{document}
{\vspace{0.01in}
	\title
	{Interior and boundary regularity results for strongly nonhomogeneous $p,q$-fractional problems}
	 
	\author{  Jacques Giacomoni$^{\,1}$ \footnote{e-mail: {\tt jacques.giacomoni@univ-pau.fr}}, \ Deepak Kumar$^{\,2}$\footnote{e-mail: {\tt deepak.kr0894@gmail.com}},  \
		and \  Konijeti Sreenadh$^{\,2}$\footnote{
			e-mail: {\tt sreenadh@maths.iitd.ac.in}} \\
		\\ $^1\,${\small Universit\'e  de Pau et des Pays de l'Adour, LMAP (UMR E2S-UPPA CNRS 5142) }\\ {\small Bat. IPRA, Avenue de l'Universit\'e F-64013 Pau, France}\\  
		$^2\,${\small Department of Mathematics, Indian Institute of Technology Delhi,}\\
		{\small	Hauz Khaz, New Delhi-110016, India} }

	\date{}
	
	\maketitle

\begin{abstract}
In this article, we deal with the global regularity of weak solutions to a class of problems involving the fractional $(p,q)$-Laplacian, denoted by $(-\Delta)^{s_1}_{p}+(-\Delta)^{s_2}_{q}$, for $s_2, s_1\in (0,1)$ and $1<p,q<\infty$. We establish completely new H\"older continuity results, up to the boundary, for the weak solutions to fractional $(p,q)$-problems involving singular as well as regular nonlinearities. Moreover, as applications to boundary estimates, we establish new Hopf type maximum principle and strong comparison principle in both situations.
%
\medskip

\noi \textbf{Key words:} Fractional $(p,q)$-Laplacian, non-homogeneous nonlocal operator, singular nonlinearity, local and boundary H\"older continuity, maximum principle, strong comparison principle.

\medskip

\noi \textit{2010 Mathematics Subject Classification:} 35J60, 35R11, 35B45, 35D30.

\end{abstract}

\section {Introduction}
 This paper aims to study the H\"older continuity results and maximum principle for weak solutions to the problems involving a class of non-homogeneous nonlocal operators. Precisely,  we consider the following  generic problem:
 \begin{equation*}
 (-\Delta)^{s_1}_{p}u+(-\Delta)^{s_2}_{q}u  = f(x) \quad \text{in } \Om,
  \tag{$\mc P$} \label{probM}
 \end{equation*}
 where $\Om\subset\mb R^N$ is a bounded domain with $C^{1,1}$ boundary, $1<q\leq p<\infty$,  $0<s_2\leq s_1<1$,  and $f\in L^\infty_{\mathrm{loc}}(\Om)$.  The fractional $p$-Laplace operator $(-\Delta)^{s}_{p}$ is defined as
 \begin{equation*}
 	{(-\Delta)^{s}_pu(x)}= 2\lim_{\vep\to 0}\int_{\mathbb R^N\setminus B_{\vep}(x)} \frac{|u(x)-u(y)|^{p-2}(u(x)-u(y))}{|x-y|^{N+ps}}dy.
 \end{equation*} 
 The operator $(-\De)_p^s$ is considered to be the natural extension of the well known fractional Laplacian, $(-\De)^s$ (i.e., $p=2$), to the nonlinear setting. 
 The leading differential operator, $(-\Delta)^{s_1}_{p}+(-\Delta)^{s_2}_{q}$, in problem $(\mc P)$ is known as the fractional $(p,q)$-Laplacian. The operator is non-homogeneous in the sense that for any $t>0$, there does not exist any $\sg\in\mb R$ such that $((-\Delta)^{s_1}_{p}+(-\Delta)^{s_2}_{q})(tu)=t^\sg ((-\Delta)^{s_1}_{p}u+(-\Delta)^{s_2}_{q}u)$ holds for all $u\in W^{s_1,p}(\Om)\cap W^{s_2,q}(\Om)$.\par
 In recent years, there has been an extensive study on the problems involving fractional Sobolev spaces and corresponding nonlocal operators due to their wide applications in the real world problems, such as game theory, 
 finance, obstacle problems, conservation laws, phase transition, image processing, anomalous diffusion and material science. For more details, we refer to \cite{bucur,caffarelli,nezzaH}  and the references therein. \par
 The fractional $(p,q)$-Laplacian is the fractional analogue of the $(p,q)$-Laplacian ($-\De_p -\De_q$), which arises in the study of general reaction-diffusion equations with nonhomogeneous diffusion and transport features. The problems involving these kinds of operators have applications in biophysics, plasma physics and chemical reactions, with double phase features, where the function $u$ corresponds to the concentration term, and the differential operator represents the diffusion coefficient,  for details see \cite{marano} and  references therein.\par 
 In the local case  (obtained by letting $s_1,s_2\to 1^-$ in \eqref{probM}), we have the H\"older continuity results for the gradient of the weak solutions. Precisely, Lieberman in \cite{lieberm}, proved $C^{1,\al}(\ov\Om)$ regularity results for quasi linear problems containing more general operators than the classical $p$-Laplacian and involving non-singular nonlinearities. Subsequently, similar regularity results were obtained in \cite{giacomoni,giacomoni2}, for the $p$-Laplace equations involving singular nonlinearities. We refer to \cite{baroni, de filippis}, for regularity results of minimizers for functional with non-standard growth.  Recently, in \cite{JDS}, the authors have obtained regularity results for weak solutions to the problems driven by the $(p,q)$-Laplacian and involving a singular nonlinearity of the form ${\rm dist}(\cdot,\pa\Om)^{-\al}u^{-\de}$, where 
  $\al\in [0,p)$, $\de>0$ and $u$ is the unknown. Here, they proved that the weak solutions are in $C^{1,\ba}(\ov\Om)$, when $\al+\de<1$ and otherwise they are in $C^{0,\ba}(\ov\Om)$, for some $\ba\in(0,1)$. See \cite{ghergu,hernand1,dpk} for details on the singular problems.\par 
 Turning to the equations involving nonlocal operators, the case of fractional Laplacian is well understood. In particular, in \cite{caffsilv}, authors obtained the interior regularity results, while Ros-Oton and Serra in \cite{ros}, proved the optimal boundary regularity. They proved that the weak solutions, for bounded right hand side, are in $ C^{s}(\mb R^N)$. For the nonlinear case ($p=q\neq 2$ and $s_1=s_2$), Iannizzotto et al. in \cite{iann}, proved that the weak solutions of problem \eqref{probM} with bounded right hand side, belong to the space $C^{0,\al}(\ov\Om)$, for some unspecified $\al\in (0,s_1]$. Here, authors established weak Harnack type inequality to obtain the interior regularity results and used the barrier arguments, in the spirit of Krylov's approach, to prove the boundary behavior of the weak solutions. In \cite{dicastro}, authors used different technique to prove the interior regularity. Here, they established the Caccioppoli inequality and Logarithmic lemma to conclude the local H\"older continuity result for the solution to the following fractional $p$-Laplacian type problems with a symmetric measurable kernel $k(x,y)\approx |x-y|^{-N-ps}$ and $g\in W^{s,p}(\mb R^N)$:
  \begin{equation*}
  	 \mc L_K u = 0 \quad \text{in } \Om, \quad
  		u=g \quad\mbox{in }\mb R^N\setminus\Om.
  \end{equation*}
 Subsequently, in \cite{brascoH}, Brasco et al. obtained the optimal H\"older interior regularity result and proved that the local weak solution $u\in W^{s,p}_{\rm loc}(\Om)$ of problem \eqref{probM} (for the case $2\leq p=q$ and $s_1=s_2=s$) is in $C^{\al}_{\rm loc}(\Om)$, for all $\al<\min\{1,ps/(p-1)\}$, when $f\in L^r_{\rm loc}(\Om)$, for suitable $r>0$. In this work, authors first obtain the improved regularity result for the case $f=0$ and then using suitable scaling argument together with decay transfer technique, they proved the similar result for non-zero $f$.  Combining the boundary behavior of the solution from \cite{iann} (obtained by using suitable barrier arguments) with the interior regularity result of \cite{brascoH}, one can get the $C^{s}$ regularity result up to the boundary. In \cite{ianndist}, Iannizzotto et al. established a slightly different regularity result. Precisely, they extended the results of \cite{ros} to the nonlinear setting and proved that the weak solution $u$ of problem \eqref{probM}, again for the case $2\leq p=q$, $s_1=s_2=s\in (0,1)$ and bounded right hand side, satisfies $\frac{u}{d^s} \in C^{\al}(\ov\Om)$, for some $\al\in (0,1)$. See \cite{kuusi}, for regularity results for nonlocal  problems involving measure data. \par 
 In this paper we also investigate the case of singular nonlinearities. We mention the former contribution in the nonlocal setting concerning homogeneous operators, that is, 
 the following prototype problem: 
 \begin{equation*}
 	\begin{array}{rl}  
 		(-\Delta)^{s}_{p}u = K(x) u^{-\de}, \; u > 0, \, \text{in }
 		\Om, \quad
 		u =0\quad \text{in } \mb R^N\setminus \Om,
 	\end{array}
 \end{equation*}
 where $K$ is a non-negative, locally bounded function which behaves like ${\rm dist}(x,\pa\Om)^{-\ba}$ near the boundary with $\ba\geq 0$ and $\de>0$.
 For the case $p=2$,  Adimurthi et al. \cite{adimurthi} obtained H\"older continuity results for the weak solutions under suitable conditions on $\ba$ and $\de$. Subsequently, Giacomoni et al. in \cite{giacomoni4}, obtained the regularity results for doubly nonlocal problems, that is, $p=2$, $\ba=0$ and with the perturbation of Choquard type term. For the case $p\neq 2$, we mention the works of 
 \cite{arora,canino}. In \cite{canino}, authors obtained certain Sobolev regularity results while in \cite{arora}, authors proved H\"older continuity results up to the boundary. Precisely, they proved that the weak  solutions are in $C^{s-\e}(\mb R^N)$, for all $\e\in(0,s)$ and for the case $p\geq 2$  with $\ba-s(1-\de)\leq 0$, whereas for the other cases, it is in $C^\sg(\mb R^N)$ for some $\sg<s$.\par
 
 Boundary estimates have natural applications as strong maximum principle and Hopf lemma. Concerning the homogeneous case, we mention the work of  Del Pezzo and Quaas in \cite{delpezzo}, where authors proved a strong maximum principle and a Hopf type lemma for continuous super-solutions to the problem:
 \begin{align*}
  (-\De)_p^s u = c(x)|u|^{p-2}u \quad\mbox{in }\Om
 \end{align*}
 where $c$ is a non-positive function and $u\geq 0$ in $\mb R^N\setminus\Om$. Later, Jarohs in \cite{jarohs}, obtained a strong comparison type result  under the condition that either the sub or super solution is bounded and is in $C^{\al}_{\rm loc}(\Om)$ with $\al(p-2)>sp-1$. Recently, in \cite{iannmospap}, authors obtained a more general form of the strong comparison principle  when the super-solution satisfies $(-\De)_p^s u + g(x)u\leq K$ weakly in $\Om$, where $K>0$ is a constant and $g$ is a continuous function. \par
 Recently, the nonlocal problems involving the fractional $(p,q)$-Laplacian   have drawn great attention of the researchers due to its non-homogeneous and nonlinear nature. The existence and multiplicity results are studied broadly but regarding the regularity results, only interior regularity is known. Here, we mention only a few such articles. The existence and multiplicity results are obtained in \cite{ambrosio1,ambrosio}, for the case of $s_1=s_2=s$, while the case $s_1\neq s_2$ is discussed in \cite{bhakta,DDS}. In \cite{ambrosio}, using variational methods, authors established concentration and multiplicity results for problems in the whole $\mb R^N$. While the bounded domain case was considered in \cite{bhakta,DDS}. Particularly, in \cite{DDS}, authors have proved the $L^\infty$ estimate and for the case of $p,q\geq 2$, they obtained the $C^{0,\al}_{\rm loc}$ regularity result, with some unspecified $\al\in (0,1)$. Further, we would like to bring attention of the reader towards the recent work of \cite{bonder}, where authors proved the global H\"older continuity results (much in the spirit of \cite{iann}) for weak solutions to problems involving the fractional $(-\De)_g^s$-Laplacian, where $g$ is a Young's function. However, this does not include our problem, even for the case $s_1=s_2$, because of the power type of growth condition (2.2) and (2.4), there.
   \par 
 Coming back to our paper, the nonhomogeneous nature of the leading operator, $(-\De)_p^{s_1}+ (-\De)_q^{s_2}$, in problem $(\mc P)$
is strengthened by use of different order exponents when $s_1\neq s_2$. This creates several difficulties as 
handling the barrier functions in order to prove the boundary behavior of the weak solutions. In \cite[Theorem 2.10]{DDS}, the interior H\"older regularity result was discussed for bounded solutions. Here,  we remove this restriction and obtain an improved result with a better and optimal H\"older exponent (see Theorem \ref{impintreg}). For this, we prove the local boundedness result of local weak solutions, which uses a new Caccioppoli type inequality (Lemma \ref{lemcaccp}) for non-homogeneous operators.  Due to lack of the scale invariance, unlike \cite{brascoH},  we directly employ Moser's iteration technique on the discrete differential of the solution and exploiting the local boundedness of the function $f$, we obtain the suitable Besov space inclusion. Further, by using embedding results for the Besov spaces to the H\"older spaces, we complete our proof of the interior regularity result, as in theorem \ref{impintreg}.  Subsequently, we establish the asymptotic behavior of the fractional $q$-Laplacian ($(-\De)_q^{s_2}$) of the distance function $d^{s_1}$ near the boundary and using this, we prove almost $d^{s_1}$ boundary behavior (which is optimal) of the weak solution, as in the proposition \ref{upper}. This coupled with the interior H\"older regularity result in Theorem \ref{impintreg} gives us the almost $s_1$-H\"older continuity result globally in $\mb R^N$. As a consequence of this, we obtain the strong maximum principle and the Hopf type maximum principle for non-negative solutions. Additionally, under the restriction that the fractional $q$-Laplacian of the subsolution is bounded from below (in the weak sense), we prove a strong comparison principle. 
In section $4$, we analyze the case of singular nonlinearities by further investigating the existence of barrier functions.  We achieve this by obtaining the existence and uniqueness of the weak solution to  the auxiliary problem \eqref{probsingeps} (defined in section \ref{singl}).  Subsequently, we establish the boundary behavior of the solutions by using the barrier functions and the local boundedness property as in proposition \ref{localbdd}. Consequently, we prove the  H\"older continuity results, up to the boundary, for minimal solutions to \eqref{probsing} and weak solutions to the critical exponent problem \eqref{probsingcrit} that were not known in  former contributions.\par
To summarize, the novelties of the paper lie in consideration of a new class of non-homogeneous nonlocal operators with different order of exponents. We establish almost optimal global H\"older continuity results for problems involving singular as well as non-singular nonlinearities. The absence of scale invariance makes the analysis involved much more delicate than its homogeneous counterpart. For instance, we can not use the barrier function as considered in \cite{iann} (for the homogeneous fractional $p$-Laplacian case) to prove the boundary behavior of the weak solutions. So, we use the distance function to construct an appropriate super solution near the boundary. Due to different order of exponent, the fractional $p$-Laplacian of $d^{s_2}$ (i.e., with the smaller exponent) is an unbounded function, which discards this as a choice for the sub-solution to prove the Hopf type maximum principle. We overcome this difficulty by using the higher exponent instead and we establish its behavior under the fractional $q$-Laplacian, as in the technical Lemma \ref{lemA2}.
Additionally, for the case of singular nonlinearities, we considered the doubly singular nonlinearity where singular weight function is also involved. In this case, the existence and regularity results extended the results of \cite{JDS} to the nonlocal framework and that of \cite{arora} to the non-homogeneous operator setting. Moreover, the strong comparison principle, as in Theorem \ref{strongcomp}, is new even for the singular problems involving fractional $p$-Laplacian.

\section{Preliminaries and Main results}
 \subsection{Notation}
 In this subsection, we fix some notations which we will use through out the paper.
 We set  $t_\pm=\max\{\pm t,0\}$.
 We denote $[t]^{p-1}:=|t|^{p-2}t$, for $p>1$ (when there is no power involved in $[\cdot]$, it should be understood as brackets), and for $S\subset\mb R^{2N}$,
 \begin{align*}
 	&A_p(u,v,S)=\ds \int_{S}\frac{[u(x)-u(y)]^{p-1}(v(x)-v(y))}{|x-y|^{N+p s_1}}~ dxdy, \text{ with} \\
 	&A_q(u,v,S)=\ds \int_{S}\frac{[u(x)-u(y)]^{q-1}(v(x)-v(y))}{|x-y|^{N+q s_2}}~ dxdy. 
 \end{align*}
 Next, for $x_0\in\mb R^N$ and $v\in L^1(B_r(x_0))$, we set \[(v)_{B_r(x_0)}:=\Xint-_{B_r(x_0)}v(x)dx=\frac{1}{|B_r(x_0)|}\int_{B_r(x_0)}v(x)dx.\]
 We define the distance function as $d(x):=\mathrm{dist}(x,\mb R^N\setminus\Omega)$, and a neighborhood of the boundary as $\Om_{\varrho}:=\{ x\in \Om \ : \ d(x)<\varrho \}$,  for $\varrho>0$.
 \subsection{Function Spaces}
 For any $E\subset\mathbb{R}^N$, $1<p<\infty$ and $0<s<1$, the fractional Sobolev space $W^{s,p}(E)$ is defined as 
 \begin{align*}
 	W^{s,p}(E):= \left\lbrace u \in L^p(E): [u]_{W^{s,p}(E)}  < \infty \right\rbrace
 \end{align*}
 \noi endowed with the norm $\|u\|_{W^{s,p}(E)}:=  \|u\|_{L^p(E)}+ [u]_{W^{s,p}(E)}$,
 where \begin{align*}
 	[u]_{W^{s,p}(E)}:=  \left(\int_{E}\int_{E} \frac{|u(x)-u(y)|^p}{|x-y|^{N+sp}}~dxdy  \right)^{1/p}.
 \end{align*} 
 Next, for any (proper) subset $E$ of $\mb R^N$, we have 
 \begin{align*}
 	W^{s,p}_0(E):=\{ u\in W^{s,p}(\mb R^N) \ : \ u=0 \quad\mbox{in }\mb R^N\setminus E \}
 \end{align*}
 which is a uniformly convex Banach space when equipped with the norm  $[\cdot]_{W^{s,p}(\mb R^N)}$ (equivalent to $\| \cdot \|_{W^{s,p}(\mb R^N)}$). Hereafter, 
 we denote the norm $[\cdot]_{W^{s,p}(\mb R^N)}$ on $W^{s,p}_0(E)$ as $\|\cdot \|_{W^{s,p}_0(E)}$.
  Moreover, we have the following relation between the spaces $W^{s_1,p}_0$ and $W^{s_2,q}_0$ (note that for the Lipschitz boundary $\pa E$, these  spaces are the same as $X_{p,s_1}$ and $X_{q,s_2}$, respectively, as defined in \cite{DDS}).
  \begin{Lemma}\cite[Lemma 2.1]{DDS}
   Let $E\subset\mb R^N$ be a bounded domain with Lipschitz boundary. Let $1<q\leq p<\infty$ and $0<s_2<s_1 <1$, then there exists a constant  $C=C(|E|,\;N,\; p,\;q,\;s_1,\;s_2)>0$ such that 
  	\begin{align*}
  	 \|u\|_{W^{s_2,q}_0(E)}\leq C \|u\|_{W^{s_1,p}_0(E)}, \quad \text{for all } \; u \in W^{s_1,p}_0(E).
  	\end{align*}
  \end{Lemma}
 Furthermore, for any bounded $E\subset\mb R^N$ and $0<s_2<s_1<1$, an easy consequence of the H\"older inequality yields 
 \[ \|u\|_{W^{s_2,q}(E)}\leq C \|u\|_{W^{s_1,p}(E)}, \quad \text{for all } \; u \in W^{s_1,p}(E). \]
 Note that, for $s_1=s_2$, the above embedding results are not true, see \cite[Theorem 1.1]{mironescu} for counter examples.
In what follows, we will consider only the case $s_1\neq s_2$ and work with the space $W^{s_1,p}$. The case $s_1=s_2=s$, runs analogously by considering the space $\mc W:= W^{s,p}(E)\cap W^{s,q}(E)$ (in place of $W^{s_1,p}(E)$),
equipped with the norm $\|\cdot\|_{W^{s,p}(E)}+\|\cdot\|_{W^{s,q}(E)}$. Similarly, we take $\mc W_0:=W^{s,p}_0(E)\cap W^{s,q}_0(E)$, equipped with the norm $\|\cdot\|_{W^{s,p}_0(E)}+\|\cdot\|_{W^{s,q}_0(E)}$.
 \begin{Definition}\label{defTail}
 Let $u:\mb R^N\to \mb R$ be a measurable function, $1<m<\infty$ and $\al>0$. We define the tail space as below:
 	\begin{align*}
 	  L^{m}_{\al}(\mb R^N) = \bigg\{ u\in L^{m}_{\rm loc}(\mb R^N) : \int_{\mb R^N} \frac{|u(x)|^{m}dx}{(1+|x|)^{N+\al}} <\infty \bigg\}.
 	\end{align*} 
  The nonlocal tail centered at $x_0\in\mb R^N$ with radius $R>0$ is defined as 
 	\begin{align*}
 	 T_{m,\al}(u;x_0,R)=\left(R^{\al}\int_{B_R(x_0)^c} \frac{|u(y)|^{m}}{|x_0-y|^{N+\al}} dy\right)^{1/m}.
 	\end{align*}
 Set $T_{m,\al}(u;R)=T_{m,\al}(u;0,R)$. We will follow the notation $T_{p-1}(u;x,R):=T_{p-1,s_1p}(u;x,R)$ and $T_{q-1}(u;x,R):=T_{q-1,s_2q}(u;x,R)$, unless otherwise stated.
 \end{Definition}
 \begin{Definition}
  Let $\Om\subset\mb R^N$ be a bounded set. Define
 	\begin{align*}
 	 \widetilde{W}^{s,p}(\Om):= \bigg\{u\in L^p_{\rm loc}(\mb R^N)\cap L^{p-1}_{sp}(\mb R^N): \exists\;U\Supset\Om \ \mbox{with }\ \|u\|_{W^{s,p}(U)} <\infty \bigg\}.
 \end{align*}\end{Definition} 
 
 For a given $g\in L^{p-1}_{s_1p}(\mb R^N)$, we define a subset of $\widetilde{W}^{s_1,p}(\Om)$ as below.
 \begin{Definition}
 Let $\Om\Subset\Om'\subset\mb R^N$ and $1<p<\infty$ with $s_1\in (0,1)$, then we define
 	\begin{align*}
 	 X^{s_1,p}_g(\Om,\Om'):= \{ v\in W^{s_1,p}(\Om')\cap L^{p-1}_{s_1p}(\mb R^N) : v=g \quad\mbox{a.e. in }\mb R^N\setminus\Om \},
 	\end{align*}
 equipped with the norm of $W^{s_1,p}(\Om')$.
 \end{Definition}
 Note that $X^{s_1,p}_g(\Om,\Om')$ is non-empty for all $g\in W^{s_1,p}(\Om')\cap L^{p-1}_{s_1p}(\mb R^N)$.
 \subsection{Statements of main results}
 In this subsection, we define the notion of a weak solution to different problems and state our main results.
 \begin{Definition}[Local weak solution]\label{defnlocal}
  A function $u\in W^{s_1,p}_{\rm loc}(\Om)\cap L^{p-1}_{s_1p}(\mb R^N) \cap L^{q-1}_{s_2q}(\mb R^N)$ is said to be a local weak solution of problem \eqref{probM} if 
 	\begin{align}\label{solndef}
 		A_p(u,\phi,\mb R^N\times \mb R^N) + A_q (u, \phi,\mb R^N\times \mb R^N)= \int_{\Om} f\phi dx,
 	\end{align}
  for all $\phi\in W^{s_1,p}(\Om)$ 
  compactly supported in $\Om$. 
 \end{Definition}
 Now, we state an improved interior H\"older regularity result for local weak solutions. 
 \begin{Theorem}\label{impintreg}
 Suppose $2\leq q\leq p<\infty$ and $0<s_2\leq s_1<1$. Let $u\in W^{s_1,p}_{\rm loc}(\Om)\cap L^{p-1}_{s_1p}(\mb R^N) \cap  L^{q-1}_{s_2q}(\mb R^N)$ be a local weak solution to problem \eqref{probM}. Then, for every $\sg\in (0,\min\{1,\frac{ps_1}{p-1}\})$, $u\in C^{0,\sg}_{\mathrm{loc}}(\Om)$. Moreover, for all $\sg\in (0,s_1)$ and $\bar R_0\in (0,1)$ with $x_0\in\mb R^N$ such that $B_{2\bar R_0}(x_0)\Subset\Om$, the following holds
 	\begin{align*}
 	 [u]_{C^{\sg}(B_{\bar R_0/2}(x_0))}\leq \big(C K_2(u)\big)^{i_\infty} \big( [u]_{W^{s_1,p}(B_{\bar R_0}(x_0))} + 1 \big)^{2-\sg/s_1}
 	\end{align*}
  where $i_\infty\in\mb N$ is such that $i_\infty>N/(s_1-\sg)-p$, $C=C(N,s_1,p,\sg)>0$ is a constant (which blows up as $\sg\to s_1^-$), and $K_2(u):=K_2(u,i_\infty,\bar R_0)$ is given by
  \begin{align*}
  	K_2(u)=1+T_{p-1}(u;x_0,\bar R_0)^{p-1} +T_{q-1}(u;x_0,\bar R_0)^{q-1} 
  	+\|u\|^{\frac{(p+i_\infty)(p-1)}{p-2}}_{L^\infty(B_{\bar R_0}(x_0))} +\|f\|_{L^\infty(B_{\bar R_0}(x_0))}.
  \end{align*} 
 \end{Theorem}
 \begin{Remark}
	We remark that the range of $\sg\in  (0,\min\{1,\frac{ps_1}{p-1}\})$ is sharp in theorem \ref{impintreg}. Indeed, as consider in \cite[Example 1.6]{brascoH}, for the case $s_1=s_2=1/2$ and $p=q=2$, the function $u(x):=\int_{B} |x-y|^{1-N}dy$, which solves $(-\De)^{1/2}u=1_B$, with $B\subset \mb R^N$ is a unit ball and $N\geq 3$, fails to be Lipschitz continuous (the quantity $\frac{ps_1}{p-1}=1$).
\end{Remark}
 
Then, we deduce:
 
\begin{Theorem}[Boundary regularity]\label{bdryreg}
 Suppose $2\leq q\leq p<\infty$ and $0<s_2\leq s_1<1$. Let $u\in  W^{s_1,p}_0(\Om)$ be a solution to problem \eqref{probM} with $f\in L^\infty(\Om)$. Then, for every $\sg\in (0,s_1)$, $u\in C^{0,\sg}(\ov\Om)$. Moreover, 
 \begin{align}\label{holderbd}
 	\| u \|_{C^\sg(\ov\Om)} \leq C,
 \end{align}
 where $C=C(\Om,N,s_1,p,\sg, \|f\|_{L^\infty(\Om)})>0$ is a constant (which blows up as $\sg\to s_1^-$ and depends as a non-decreasing function of $\|f\|_{L^\infty(\Om)}$).
\end{Theorem}
As a direct consequence, we have:
\begin{Corollary}\label{cornonpb}
 Let $u\in W^{s_1,p}_0(\Om)$ be a solution to problem \eqref{probM} with $f(x):=f(x,u)$, a Carath\'eodory function such that $|f(x,t)|\leq C_0 (1+|t|^{p^*_{s_1}-1})$, where $C_0 >0$, and $p^*_{s_1}:=Np/(N-ps_1)$ if $N>ps_1$, otherwise an arbitrarily large number. Then, $u\in C^{0,\sg}(\ov\Om)$, for all $\sg\in (0,s_1)$, and \eqref{holderbd} holds. 
\end{Corollary}

 \begin{Remark}
	The regularity results for the case $1<q\leq p<2$ and $1<q <2 \leq p$ are still open. The crucial step is the interior regularity. Indeed, we note that by establishing the Logarithmic lemma (similar to \cite[Lemma 1.3]{dicastro}) and using the Caccioppoli inequality \eqref{caccioppnew}, one can prove the interior regularity result for weak solutions of \eqref{probM} with bounded $f$. This coupled with the boundary behavior, given by proposition \ref{uppdist}, implies that the solutions are in $C^{0,\al}(\ov\Om)$, for some $\al\in (0,1)$. 
\end{Remark}

Now we state our strong maximum principle for non-negative continuous super-solutions.
\begin{Theorem}[Strong maximum principle]\label{strongmax}
  Suppose $2\leq q\leq p<\infty$ and $0<s_2\leq s_1<1$. Let $u\in  W^{s_1,p}_0(\Om)\cap C(\ov\Om)$ be a weak super-solution of $(-\De)_{p}^{s_1} u+  (-\De)_q^{s_2} u=0$ in $\Om$ with $u\geq 0$ in $\mb R^N\setminus\Om$. Then, either $u= 0$ in $\mb R^N$ or $u>0$ in $\Om$.
\end{Theorem}
 Next, we have the following Hopf type maximum principle.
 \begin{Proposition}\label{hopf}
  Suppose the hypotheses of Theorem \ref{strongmax} are true in addition to $s_1\neq q's_2$.   Then, either $u= 0$ in $\mb R^N$ or 
	\begin{align*}
	 \liminf_{\Om\ni x\to x_0} \frac{ u(x)}{d(x)^{s_1}}>0 \quad\mbox{for all } x_0\in\pa\Om.
	\end{align*}
\end{Proposition} 

From above, we derive:
\begin{Theorem}[Strong Comparison principle]\label{strongcompreg}
 Let $u,v\in W^{s_1,p}_0(\Om)\cap C(\ov\Om)$ be such that $0<v\leq u$ in $\Om$ with $u\not\equiv v$. Further, assume that for some $K, K_1>0$, the following holds:
	\begin{equation*}
	 \left\{\begin{array}{rllll}
		(-\De)_p^{s_1}v +(-\De)_q^{s_2}v &\leq (-\De)_p^{s_1}u +(-\De)_q^{s_2}u \leq K,  \\
		-K_1 &\leq (-\De)_q^{s_2} v,
		\end{array}
		\right.
		\quad\mbox{weakly in }\Om.
	\end{equation*}
 Then $u>v$ in $\Om$. Moreover, for $s_1\neq q's_2$, $\frac{u-v}{d^{s_1}}\geq C>0$ in $\Om$.
\end{Theorem} 
 \begin{Remark}
	For the case $s_1= q's_2$, we can prove the similar result as  Proposition \ref{hopf} and Theorem \ref{strongcompreg}, but with $d^\alpha$ instead of $d^{s_1}$, where $\alpha\geq s_1$ is such that $\alpha\neq q's_2$ and $\alpha\neq p's_1$. Moreover, for the case $s_1\neq q's_2$, the exponent $s_1$ in the function $d^{s_1}$ is optimal.
\end{Remark}

Now, we turn to the case of singular nonlinearity.  We consider the following problem:  
\begin{equation*}
	\left\{\begin{array}{rllll}
		(-\Delta)^{s_1}_{p}u+(-\Delta)^{s_2}_{q}u & = K_{\ga}(x)u^{-\de}, \; \ u>0 \; \text{ in } \Om, \\ 
		u&=0 \quad \text{in } \mathbb{R}^N\setminus \Om
	\end{array}
	\right. \tag{$\mc S_{\ga,\de}$}\label{probsing}
\end{equation*}
where $\de>0$, $\ga\in [0,ps_1)$ and $K_\ga\in L^\infty_{\mathrm{loc}}(\Om)$ is such that 
\begin{align}\label{kermain} 
 \mc C_1 d(x)^{-\ga} \leq K_\ga(x)\leq \mc C_2 d(x)^{-\ga} \quad\mbox{in }\Om, 
\end{align}
with $\mc C_1, \mc C_2$ are positive constants. We have the following notion of solutions:
\begin{Definition}\label{defn1}
 A function $u\in W^{s_1,p}_{\rm loc}(\Om)$ is said to be a weak solution to problem \eqref{probsing} if there exists $\theta\geq 1$ such that $u^\theta\in W^{s_1,p}_0(\Om)$, $\inf_D u >0$ for all $D\Subset\Om$, and \eqref{solndef} holds for all $\phi\in C_c^\infty(\Om)$ with $f(x):= K_\ga(x)u^{-\de}(x)$, there.
\end{Definition}
\begin{Remark}\label{remequivsoln}
 We note that a solution $u\in W_{\rm loc}^{s_1,p}(\Om)$ of \eqref{probsing} in the sense of definition \ref{defn1} is in fact a local solution in the sense of definition \ref{defnlocal}.
\end{Remark}

We prove in the singular setting the following almost optimal global regularity results:
 \begin{Theorem}[Boundary regularity]\label{bdryregsing}
  Let $2\leq q\leq p<\infty$, $\ga\in[0,ps_1)$ and $\de>0$. Let $v\in W^{s_1,p}_{\rm loc}(\Om)$ be the minimal solution to problem \eqref{probsing}. The following holds: 
  \begin{itemize}
  	\item[(i)] if $\ga-s_1(1-\de)\leq 0$, then $v\in C^{0,\sg}(\ov\Om)$, for all $\sg<s_1$. 
  	\item[(ii)] If $\ga-s_1(1-\de)> 0$ with $\ga\neq ps_1-q's_2(p-1+\de)$, then $v\in C^{0,\frac{ps_1-\ga}{p-1+\de}}(\ov\Om)$.
  	\item[(iii)] If $\ga-s_1(1-\de)> 0$ with $\ga= ps_1-q's_2(p-1+\de)$, then $v\in C^{0,\frac{ps_1-\ga_1}{p-1+\de}}(\ov\Om)$, for all $\ga_1\in (\ga,ps_1)$.
  \end{itemize}
 \end{Theorem}
For critical growth problems, we deduce:
\begin{Corollary}\label{corcriticalsing}
 Let $u\in W^{s_1,p}_0(\Om)$ be a weak solution to problem:
 \begin{equation*}
  \left\{
 	\begin{array}{rl}  
 		(-\Delta)^{s_1}_{p}u+  (-\Delta)^{s_2}_{q}u &= \la u^{-\de}+ b(x,u), \quad 
 		u > 0 \quad \mbox{in } \Om,\\
 		u &=0\quad \text{in } \mb R^N\setminus \Om,
 	\end{array}
 	\right. \tag{$\mc Q_\la$} \label{probsingcrit}
 \end{equation*}
 where $\de\in (0,1)$, $\la>0$ is a parameter and $b:\Om\times\mb R\to\mb R$ is a Carath\'eodory function satisfying $|b(x,t)|\leq C_b (1+|t|^{p^*_{s_1}-1})$ with $C_b>0$ as a constant.  Then, $u\in C^{0,\sg}(\ov\Om)$, for all $\sg\in (0,s_1)$.
\end{Corollary}
%

Next, we establish a version of the strong comparison principle (weaker than theorem \ref{strongcompreg}) for problems involving singular nonlinearities.
\begin{Theorem}\label{strongcomp}
 Let $2\leq q\leq p<\infty$ and $g\in L^\infty(\Om)$ be a non-negative function. Let $v,w\in W^{s_1,p}_0(\Om)\cap L^\infty(\Om)$ satisfy the following:
	\begin{align*}
	 (-\Delta)^{s_1}_{p}v+  (-\Delta)^{s_2}_{q}v \ge v^{-\de}+ g(x), \quad\mbox{and }
	 (-\Delta)^{s_1}_{p}w+  (-\Delta)^{s_2}_{q}w \le w^{-\de}+ g(x). 
	\end{align*} 
  Furthermore, assume either $v$ or $w$ is in $C^{0,\al}_{\mathrm{loc}}(\Om)$ with $\al>\max\{ \frac{ps_1-1}{p-2}, \frac{qs_2-1}{q-2} \}$ and there exists $\eta>0$ such that $v\geq \eta d^{s_1}$ in $\Om$. Then, either $v=w$ in $\mb R^N$ or 
	\[ \mathrm{inf}_K(v-w)>0 \quad \mbox{for all }K\Subset\Om. \]	
\end{Theorem}

 \section{Regularity results for regular problems}
In this section, we study the regularity results for problem \eqref{probM}. First, we prove the interior bounds on the weak solutions. For $f\in L^\infty_{\mathrm{loc}}(\Om)$ and $g\in L^{p-1}_{s_1p}(\mb R^N)\cap L^{q-1}_{s_2q}(\mb R^N)$, we consider the following boundary value problem:
 \begin{equation*}
 	\left\{
 	\begin{array}{rl}
 		(-\Delta)^{s_1}_{p}u+  (-\Delta)^{s_2}_{q}u &= f \quad \text{in }\Om, \\
 		u &=g\quad \text{in} \ \mb R^N\setminus \Om.
 	\end{array}
 	\right. \tag{$\mc G_{f,g}(\Om)$} \label{probG}
 \end{equation*}
 \begin{Definition}\label{defnsolnbdry}
 For $f\in L^\infty_{\mathrm{loc}}(\Om)$ and $g\in L^{p-1}_{s_1p}(\mb R^N)\cap L^{q-1}_{s_2q}(\mb R^N)$, we say that $u\in X^{s_1,p}_{g}(\Om,\Om') \cap X^{s_2,q}_{g}(\Om,\Om')$ is a solution to problem $(\mc G_{f,g}(\Om))$ if  \eqref{solndef} holds for all $\phi\in X^{s_1,p}_{0}(\Om,\Om') \cap X^{s_2,q}_{0}(\Om,\Om')$.
 \end{Definition}
Now, we prove the following non-local Caccioppoli type inequality. 

\begin{Lemma}\label{lemcaccp}
Suppose $1<q\leq p<\infty$. Let $\Om''\Subset\Om'\subset\mb R^N$ and let $u\in X^{s_1,p}_{g}(\Om'',\Om') \cap X^{s_2,q}_{g}(\Om'',\Om')$ be a solution to problem $(\mc G_{f,g}(\Om''))$ with $g\in W^{s_1,p}(\Om')\cap L^{p-1}_{s_1p}(\mb R^N)\cap L^{q-1}_{s_2q}(\mb R^N)$. Set $w_{\pm}(x):=(u(x)-k)_{\pm}$, for $k\in\mb R$. Then, for any $E\Subset\Om''$ and $\psi\in C_c^\infty(E)$ with $0\leq\psi\leq 1$, we have 
 \begin{align}\label{caccioppnew}
 	&\int_{E} \int_{E} \frac{|w_\pm(x)\psi(x)-w_\pm(y)\psi(y)|^p}{|x-y|^{N+s_1p}}~dxdy \nonumber\\ 
 	&\leq C \sum_{(l,s)} \int_{E}\int_{E} \frac{|\psi(x)-\psi(y)|^l}{|x-y|^{N+sl}}(w_\pm(x)^l+w_\pm(y)^l)dxdy \nonumber\\
 	&\quad+ C   \sum_{(l,s)} \Big(\sup_{y\in {\rm supp}(\psi)} \int_{\mb R^N\setminus E} \frac{w_\pm(x)^{l-1}dx}{|x-y|^{N+sl}}\Big) \int_{E}w_\pm\psi^p+ C \int_{\Om''} |f| w_\pm\psi^p,
 \end{align}  
 where $(l,s)\in\{(p,s_1), (q,s_2)\}$ and $C=C(p,q)>0$ is a constant.
\end{Lemma}
 \begin{proof}
   Taking  $\phi=w_+\psi^p$ in \eqref{solndef} (observe that $u\mp w_\pm\psi^p\in X^{s_1,p}_{g}(\Om'',\Om') \cap X^{s_2,q}_{g}(\Om'',\Om')$), we have
   \begin{align*}
  	A_p(u,w_+\psi^p,\mb R^N\times \mb R^N)+ A_q(u,w_+\psi^p,\mb R^N\times \mb R^N)=\int_{\Om} f(x) w_+(x)\psi^p(x)dx.
  \end{align*}
 The above expression can be simplified as below:
  \begin{align}\label{eq130}
  	&\sum_{(l,s)} \int_{E}\int_{E} \frac{[u(x)-u(y)]^{l-1} (w_+(x)\psi(x)^p-w_+(y)\psi(y)^p) }{|x-y|^{N+ls}} dxdy \nonumber\\
  	 &+2\sum_{(l,s)} \int_{\mb R^N\setminus E}\int_{E} \frac{[u(x)-u(y)]^{l-1}w_+(x)\psi(x)^p}{|x-y|^{N+ls}} dxdy \nonumber\\  
  	  &=\int_{\Om} f(x) w_+(x)\psi(x)^p dx,
  \end{align}
  where $(l,s)\in \{ (p,s_1), (q,s_2)\}$. To estimate the first integral, without loss of generality, we assume $u(x)\geq u(y)$, consequently
    {\small\begin{align*}
     [u(x)-u(y)]^{l-1}& (w_+(x)\psi(x)^p-w_+(y)\psi(y)^p)=(u(x)-u(y))^{l-1} ((u(x)-k)_+\psi(x)^p-(u(y)-k)_+\psi(y)^p) \\
      &=\begin{cases}
      	(w_+(x)-w_+(y))^{l-1}(w_+(x)\psi(x)^p-w_+(y)\psi(y)^p) \quad\mbox{if }u(x), u(y)>k, \\
      	(u(x)-u(y))^{l-1}w_+(x)\psi(x)^p \quad\mbox{if }u(x)>k, u(y)\leq k, \\
      	0 \quad\mbox{otherwise}.
      \end{cases} 
    \end{align*} }
  Thus, we conclude that 
    {\small\begin{align}\label{eq131}
     [u(x)-u(y)]^{l-1} (w_+(x)\psi(x)^p-w_+(y)\psi(y)^p) \geq [w_+(x)-w_+(y)]^{l-1}(w_+(x)\psi(x)^p-w_+(y)\psi(y)^p).
    \end{align}}
  For the second integral in \eqref{eq130}, we have 
    \begin{align*}
     [u(x)-u(y)]^{l-1}w_+(x)\geq - (u(y)-u(x))^{l-1}(u(x)-k)_+\geq -w_+(y)^{l-1}w_+(x).
    \end{align*}
 This implies that 
  {\small\begin{align}\label{eq132}
  	\int_{\mb R^N\setminus E}\int_{E} \frac{[u(x)-u(y)]^{l-1}w_+(x)\psi(x)^p}{|x-y|^{N+ls}} dxdy &\geq -\int_{\mb R^N\setminus E}\int_{E} \frac{w_+(y)^{l-1}w_+(x)\psi(x)^p}{|x-y|^{N+ls}} dxdy \nonumber\\
  	&\geq -\left[\sup_{x\in {\rm supp}(\psi)} \int_{\mb R^N\setminus E} \frac{w_+(y)^{l-1}dy}{|x-y|^{N+ls}}\right] \int_{E}w_+(x)\psi^p(x)dx.
  \end{align}}
  Therefore, using \eqref{eq131} and \eqref{eq132} in \eqref{eq130}, we infer that 
  \begin{align}\label{eq134}
  	&\sum_{(l,s)} \int_{E}\int_{E} \frac{[w_+(x)-w_+(y)]^{l-1} (w_+(x)\psi(x)^p-w_+(y)\psi(y)^p) }{|x-y|^{N+ls}} dxdy \nonumber\\
  	&\leq 2\sum_{(l,s)}  \Big(\sup_{x\in {\rm supp}(\psi)} \int_{\mb R^N\setminus E} \frac{w_+(y)^{l-1}dy}{|x-y|^{N+ls}}\Big) \int_{E}w_+(x)\psi^p(x)dx
  	+\int_{\Om''} |f(x)| w_+(x)\psi^p(x)dx.
  \end{align}
 We estimate the left hand side term as below:
    {\small\begin{align*}
     (w_+(x)\psi(x)^p-w_+(y)\psi(y)^p)=\frac{w_+(x)-w_+(y)}{2} (\psi(x)^p+\psi(y)^p)+ \frac{w_+(x)+w_+(y)}{2} (\psi(x)^p-\psi(y)^p),
    \end{align*}}
 consequently,
  \begin{align}\label{eq133}
  	[w_+(x)-w_+(y)]^{l-1}& (w_+(x)\psi(x)^p-w_+(y)\psi(y)^p) \nonumber\\
  	&\geq |w_+(x)-w_+(y)|^{l}\frac{\psi(x)^p+\psi(y)^p}{2} \nonumber\\
  	&\quad- |w_+(x)-w_+(y)|^{l-1}\frac{w_+(x)+w_+(y)}{2}|\psi(x)^p-\psi(y)^p|.
  \end{align}
 Observing that 
 \begin{align*}
 	|\psi(x)^p-\psi(y)^p| \leq p(\psi(x)^p+\psi(y)^p)^{(p-1)/p} |\psi(x)-\psi(y)|,
 \end{align*}
 and using Young's inequality,  we get for $\e>0$ small,
  \begin{align*}
  	|w_+(x)-w_+(y)|^{l-1}&(w_+(x)+w_+(y))|\psi(x)^p-\psi(y)^p| \nonumber\\
  	&\leq \e |w_+(x)-w_+(y)|^{l} (\psi(x)^p+\psi(y)^p)  \nonumber\\
  	&\quad+  C(\e) (\psi(x)^p+\psi(y)^p)^{(p-l)/p} |\psi(x)-\psi(y)|^{l}(w_+(x)+w_+(y))^l.
  \end{align*}
 Thus, from \eqref{eq133} and using the fact that $\psi\leq 1$, for sufficiently small $\e>0$, we get
 \begin{align*}
 	[w_+(x)-w_+(y)]^{l-1}& (w_+(x)\psi(x)^p-w_+(y)\psi(y)^p) \nonumber\\
 	&\geq |w_+(x)-w_+(y)|^{l}\frac{\psi(x)^p+\psi(y)^p}{4} \nonumber\\
 	&\quad- C|\psi(x)-\psi(y)|^{l}(w_+(x)+w_+(y))^l. 
 \end{align*}
 Therefore, on account of \eqref{eq134}, we deduce that
  \begin{align}\label{eq139}
  	& \int_{E}\int_{E} \frac{|w_+(x)-w_+(y)|^{p} }{|x-y|^{N+ps_1}}(\psi(x)^p+\psi(y)^p) dxdy \nonumber\\
  	&\leq C\sum_{(l,s)}\int_{E}\int_{E} \frac{|\psi(x)-\psi(y)|^{l} }{|x-y|^{N+ls}}(w_+(x)+w_+(y))^l dxdy \nonumber\\
  	&+ C\sum_{(l,s)}  \Big(\sup_{x\in {\rm supp}(\psi)} \int_{\mb R^N\setminus E} \frac{w_+(y)^{l-1}dy}{|x-y|^{N+ls}}\Big) \int_{E}w_+(x)\psi^p(x)dx
  	+\int_{\Om''} |f(x)| w_+(x)\psi^p(x)dx.
  \end{align}
 Next, 
 \begin{align*}
  |w_+(x)\psi(x)^p - w_+(y)\psi(y)^p| &\leq 2^{p-1} |w_+(x) - w_+(y)|^p (\psi(x)^p+\psi(y)^p) \nonumber \\
  &\quad+ 2^{p-1} |\psi(x)-\psi(y)|^{p}(w_+(x)+w_+(y))^p,
 \end{align*}
 this together with \eqref{eq139} completes the proof of \eqref{caccioppnew} for $w_+$. \\
 The proof for $w_-$ runs along the similar lines of $w_+$, with minor modifications, by taking $\tl\phi= -w_-\psi^p$ in place of $\phi$. \QED
 \end{proof}

\begin{Proposition}[Local boundedness]\label{localbdd}
 Let the hypotheses of Lemma \ref{lemcaccp} be true. Let $r\in (0,1)$ and $x_0\in\mb R^N$ be such that $B_r(x_0)\Subset\Om''$. Then, the following holds
  \begin{align}\label{eqlocalbd}
  	\sup_{B_{r/2}(x_0)} w_\pm \leq C \left( \Xint-_{B_r}w_\pm^p dx \right)^{1/p}+ T_{p-1}(w_\pm;x_0,\frac{r}{2})+T_{q-1}(w_\pm;x_0,\frac{r}{2})^\frac{q-1}{p-1}+ 1, 
  \end{align}
 where $C=\big(C(N,p,q,s_1,s_2)(1+\|f\|_{L^\infty(B_r)})\big)^{p^*_{s_1}/p^2}$, with $C(N,p,q,s_1,s_2)>0$ as a constant.
\end{Proposition}
\begin{proof}
 We proceed similar to \cite[Theorem 1.4]{dicastro}.  For fixed $\tl k\in\mb R^+$ and $k\in\mb R$, we define the following: 
  \begin{align*}
  	r_i= (1+2^{-i})\frac{r}{2}, \ \ \tl r_i=\frac{r_i+r_{i+1}}{2};  	 \ \ k_i=k+(1-2^{-i})\tl k, \ \ \tl k_i= \frac{k_{i}+k_{i+1}}{2} \quad\mbox{for all }i\in\mb N.  
  \end{align*}
  Clearly, $r_{i+1}\leq \tl r_i \leq r_i$ and $k_i\leq \tl k_i$. Further, we set
  \begin{align*}
  	&B_i= B_{r_i}(x_0), \ \ \tl B_{r_i}= B_{\tl r_i}(x_0); \\
  	&\psi_i\in C^\infty_c(\tl B_i), \ 0\leq \psi_i \leq 1, \ \psi_i\equiv 1 \mbox{ in }B_{i+1}, \ \ |\na\psi_i| < 2^{i+3}/r; \\
    &\tl w_i= (u-\tl k_i)_+ \quad\mbox{and } w_i=(u-k_i)_+.
  \end{align*}
 Using the fractional Poincar\'e inequality for $\tl w_i\psi_i$ (for the case $ps_1<N$, otherwise taking an arbitrarily large number in place of $p^*_{s_1})$, we get 
 \begin{align}\label{eq135}
 	\bigg| \left(\Xint-_{B_i} |\tl w_i\psi_i|^{p^*_{s_1}}\right)^\frac{1}{p^*_{s_1}} - \big| \Xint-_{B_i} \tl w_i\psi_i\big|  \bigg|^p &\leq c \left(  \Xint-_{B_i} |\tl w_i\psi_i- (\tl w_i\psi_i)_{B_i}|^{p^*_{s_1}}\right)^\frac{p}{p^*_{s_1}} \nonumber\\
 	&\leq c\frac{r^{s_1p}}{r^N}\int_{B_i}\int_{B_i} \frac{|\tl w_i(x)\psi_i(x)-\tl w_i(y)\psi_i(y)|^p}{|x-y|^{N+ps_1}}dxdy.
 \end{align}
 Employing the Caccioppoli inequality \eqref{caccioppnew}, for $w_+=\tl w_i$, $\psi=\psi_i$ and $E=B_i$, we obtain
 {\small\begin{align}\label{eq136}
 	\int_{B_i}\int_{B_i} \frac{|\tl w_i(x)\psi_i(x)-\tl w_i(y)\psi_i(y)|^p}{|x-y|^{N+ps_1}}dxdy 
 	&\leq C\left[ \sum_{(l,s)} \int_{B_i}\int_{B_i} \frac{|\psi_i(x)-\psi_i(y)|^l}{|x-y|^{N+sl}}(\tl w_i(x)^l+ \tl w_i(y)^l)dxdy \nonumber \right.\\
 	& \left. +
 	\sum_{(l,s)}  \Big(\sup_{y\in {\rm supp}(\psi_i)} \int_{\mb R^N\setminus B_i} \frac{\tl w_i(x)^{l-1}dx}{|x-y|^{N+ls}}\Big) \int_{B_i}\tl w_i(y)\psi_i^p(y)dy \right. \nonumber \\
 	&\left. +\int_{\Om} |f(x)| \tl w_i(x)\psi_i^p(x)dx \right].
 \end{align} }
 Now we estimate the various terms of \eqref{eq136} in the following steps.\\
 \textit{Step I}: Estimate of the first term in \eqref{eq136}.\\
 By noticing the bounds on $\psi_i$ and the relation $\tl w_i^l \leq \frac{w_i^p}{(\tl k_i -k_i)^{p-l}}$, we have 
 \begin{align*}
 	\int_{B_i}\int_{B_i} \frac{|\psi_i(x)-\psi_i(y)|^l}{|x-y|^{N+sl}}(\tl w_i(x)^l+ \tl w_i(y)^l)dxdy 
 	&\leq c \frac{2^{i}}{r^l} \int_{B_i}\frac{w_i^p(y)}{(\tl k_i -k_i)^{p-l}}\left(\int_{B_i} \frac{dx}{|x-y|^{N+sl-l}}\right)dy \\
 	&\leq c \frac{2^{i(l+p-l)}}{l(1-s)}\frac{1}{\tl k^{p-l}}\frac{r^N}{r^{sl}} \Xint-_{B_i}w_i^p(x)dx.
 \end{align*}
 \textit{Step II}: Estimate of the second term in \eqref{eq136}.\\
 For $y\in\tl B_i={\rm supp}(\psi_i)$ and $x\in\mb R^N\setminus B_i$, we have 
 \begin{align*}
 	\frac{|x-x_0|}{|x-y|}\leq 1+ \frac{|x-x_0|}{|x-y|}\leq 1+\frac{\tl r_i}{r_i - \tl r_i}\leq 2^{i+4}.
 \end{align*}
 Then, using the relation $\tl w_i \leq \frac{w_i^p}{(\tl k_i -k_i)^{p-1}}$, we obtain 
  \begin{align*}
  	\Big(\sup_{y\in {\rm supp}(\psi_i)} &\int_{\mb R^N\setminus B_i} \frac{\tl w_i(x)^{l-1}dx}{|x-y|^{N+ls}}\Big) \int_{B_i}\tl w_i(y)\psi_i^p(y)dy \\
  	&\leq c \frac{2^{i(N+sl)}}{(\tl k_i -k_i)^{p-1}} \left(\int_{B_i} w_i^p(y)dy \right) \left( \int_{\mb R^N\setminus B_i} \frac{\tl w_i(x)^{l-1}dx}{|x-x_0|^{N+ls}} \right) \\
  	&\leq c\frac{2^{i(N+sl+p-1)}}{\tl k^{p-1}} \frac{r^N}{r^{sl}} T_{l-1}(w_0;x_0,r/2)^{l-1} \Xint-_{B_i} w_i^p(y)dy.
  \end{align*}
 \textit{Step III}: Estimate of the third term in \eqref{eq136}.\\
 Proceeding similarly to step II, we obtain
 \begin{align*}
 	\int_{\Om} |f(x)| \tl w_i(x)\psi_i^p(x)dx \leq c\frac{2^{i(p-1)}}{\tl k^{p-1}}  \| f \|_{L^\infty(B_r)} r^N \Xint-_{B_i} w_i^p(y)dy.
 \end{align*}
 Combining the estimate of Steps I, II and III in \eqref{eq136} and using \eqref{eq135}, we get
 {\small\begin{align*}
 	\left(\Xint-_{B_i} |\tl w_i\psi_i|^{p^*_{s_1}}\right)^\frac{1}{p^*_{s_1}} \leq  c 2^{i(N+sl+p-1)} \left[ \frac{1+\tl k^{q-p}}{q(1-s_1)}+ \frac{\sum_l T_{l-1}(w_0;x_0,\frac{r}{2})^{l-1}}{\tl k^{p-1}} + \| f \|_{L^\infty(B_r)} + 1 \right] \Xint-_{B_i} w_i^p(y)dy.
 \end{align*}}
 Now, using \cite[(4.7), p. 1292]{dicastro} together with the notation $A_i:=\big(\Xint-_{B_i} w_i^p(y)dy\big)^{1/p}$, in the above expression, we obtain
 {\small \begin{align}\label{eq137}
 	\left(\frac{\tl k}{2^{i+2}}\right)^\frac{p(p^*_{s_1}-p)}{p^*_{s_1}} A_{i+1}^\frac{p^2}{p^*_{s_1}} \leq c 2^{i(N+sp+p-1)} \left[ \frac{1+\tl k^{q-p}}{q(1-s_1)}+ \frac{\sum_l T_{l-1}(w_0;x_0,\frac{r}{2})^{l-1}}{\tl k^{p-1}} + \| f \|_{L^\infty(B_r)} + 1 \right] A_i^p.
 \end{align}}
 Then, we choose 
  \[ \tl k \geq T_{p-1}(w_0;x_0,\frac{r}{2})+T_{q-1}(w_0;x_0,\frac{r}{2})^\frac{q-1}{p-1}+ 1. \]
 Therefore, from \eqref{eq137}, we infer that
  \begin{align*}
  	\left(\frac{A_{i+1}}{\tl k}\right)^\frac{p}{p^*_{s_1}} \leq c^\frac{p}{p^*_{s_1}} 2^{i(\frac{N+sp+p-1}{p}+\frac{s_1p}{N})} \frac{A_i}{\tl k},
  \end{align*}
 where $c=\big(c(p,q,s_1,N)(1+\|f\|_{L^\infty(B_r)})\big)^{p^*_{s_1}/p^2}>0$. Simplifying the above expression for $\ga=s_1p/(N-s_1p)=p^*_{s_1}/p-1$ and $C=2^{\frac{(N+sp+p-1)N}{p(N-s_1p)}+\frac{s_1p}{N-s_1p}}>1$, we obtain
  \begin{align*}
  	\frac{A_{i+1}}{\tl k} \leq c C^i  \left(\frac{A_{i}}{\tl k}\right)^{1+\ga}.
  \end{align*}
 Now proceeding similarly to \cite[Proof of Theorem 1.1, p. 1293]{dicastro} with the choice 
  \[ \tl k =T_{p-1}(w_0;x_0,\frac{r}{2})+T_{q-1}(w_0;x_0,\frac{r}{2})^\frac{q-1}{p-1}+ 1 + c^\frac{1}{\ga}C^\frac{1}{\ga^2} A_0, \]
  we get
  \begin{align*}
  	\sup_{B_{r/2}} w_+\leq \tl k=  T_{p-1}(w_+;x_0,\frac{r}{2})+T_{q-1}(w_+;x_0,\frac{r}{2})^\frac{q-1}{p-1}+ 1 + c^\frac{1}{\ga}C^\frac{1}{\ga^2} \left( \Xint-_{B_r}w_+^p \right)^\frac{1}{p}.
  \end{align*}
 An analogous treatment for $w_-$ yields the similar result. This completes the proof of the proposition. \QED
\end{proof}
\begin{Corollary}\label{corlocalbdd}
 Suppose $1<q\leq p<\infty$, $s_2\leq s_1\in (0,1)$.  Let $u\in W^{s_1,p}_{\rm loc}(\Om)\cap L^{p-1}_{s_1p}(\mb R^N) \cap L^{q-1}_{s_2q}(\mb R^N)$ be a local weak solution to problem \eqref{probM}. Then, $u\in L^\infty_{\mathrm{loc}}(\Om)$. 
\end{Corollary}
\begin{proof}
 Let $x_0\in\mb R^N$ and $R\in (0,1)$ such that $B_{2R}(x_0)\Subset\Om$. Since $u\in W^{s_1,p}_{\rm loc}(\Om)$, we have that $u\in W^{s_1,p}(B_{2R}(x_0))$ and satisfies:
  \begin{equation*}
  	\left\{ \begin{array}{rlll}
  	  (-\Delta)^{s_1}_{p}u+  (-\Delta)^{s_2}_{q}u &= f \quad \text{in}\;
  	 B_{R}(x_0), \\
  	   u &=u\quad \text{in} \ \mb R^N\setminus B_{R}(x_0).	
  	\end{array}
    \right.
  \end{equation*}
 Then, applying proposition \ref{localbdd} to $u$, with $k=0$ in there, we get $u\in L^\infty(B_{R/2}(x_0))$. By a covering argument, we get that $u\in L^\infty_{\mathrm{loc}}(\Om)$ and the estimate of the proposition \ref{localbdd} holds for $u$ (in place of $w_\pm$), and for all $r\in (0,1)$ and $x_0\in\mb R^N$ such that $B_{2r}(x_0)\Subset\Om$. \QED
\end{proof}

Next, we mention that the solution $u$ of problem $(\mc G_{f,g}(\Om''))$ remains bounded near the boundary also, provided $g$ is bounded. The proof runs exactly the same as that of proposition \ref{localbdd} by using the corresponding Caccioppoli type inequality near the boundary (similar to \eqref{caccioppnew}, obtained from Lemma \ref{caccioppnew} for any $E\subset\Om''\Subset\Om'$, by noting that $u\in W^{s_1,p}(\Om')\cap W^{s_2,q}(\Om')$).
\begin{Proposition}
 Let $\Om''\Subset\Om'\subset\mb R^N$ and let $u\in X^{s_1,p}_{g}(\Om'',\Om') \cap X^{s_2,q}_{g}(\Om'',\Om')$ be a solution to problem $(\mc G_{f,g}(\Om''))$ with $g\in W^{s_1,p}(\Om')\cap L^{p-1}_{s_1p}(\mb R^N)\cap L^{q-1}_{s_2q}(\mb R^N)$. Let $x_0\in\pa\Om''$ and let $r\in (0,1)$ be such that $B_r(x_0)\Subset\Om'$, and suppose $g\in L^\infty(B_r(x_0))$.
 Set $w_{\pm}(x):=(u(x)-k_\pm)_{\pm}$, where $k_+>\sup_{B_r(x_0)} g$ and $k_-<\inf_{B_r(x_0)} g$. Then, \eqref{eqlocalbd} holds, consequently, $u$ is bounded in $B_{r/2}(x_0)$.
\end{Proposition}
Next, we state the boundedness property of solutions to the critical exponent problems.  
\begin{Theorem}\label{mainbound}
 Let $1<q\leq p<\infty$ and $0<s_2\leq s_1<1$.  Let $u\in W^{s_1,p}_0(\Om)$ be a weak solution to problem \eqref{probM} with $f(x):=f(x,u)$, a Carath\'eodory function satisfying $|f(x,t)|\leq C_0(1+|t|^{p^*_{s_1}-1})$, for all $t\in\mb R$ and a.e. $x\in\Om$, where $C_0>0$ is a constant. Then, $u\in L^\infty(\Om)$. 
\end{Theorem}
\begin{proof}
 We sketch the proof, as it is essentially the same as \cite[Theorem 3.3]{chenJFA}. For $m>0$, $t\ge 0$ and $\kappa\ge 0$, we define $g_\ka(t)=t (\llcorner t \lrcorner_m)^\ka$, where $\llcorner t \lrcorner_m:=\min\{t,m\}$ and extend $g_\ka$ and $\llcorner t \lrcorner_m$ as an odd function. Set $G_\ka(t)=\int_{0}^{t}g^\prime_\ka(\tau)^{1/p}d\tau$. Then, it is easy to verify that 
 \begin{align*}
 	G_\ka(t)\geq \frac{p(\ka+1)^{1/p}}{\ka+p} g_{\ka/p}(t).
 \end{align*}
 Taking $g_\ka(u)$ as a test function in the weak formulation of problem \eqref{probM}, we get
 \begin{align}\label{eq21}
  A_p(u, g_\ka(u),\mb R^{2N})+ A_q(u, g_\ka(u),\mb R^{2N}) &=\int_{\Om}f(x,u)g_\ka(u)\leq C_0 \int_{\Om} \big(1+|u|^{p^*_{s_1}-1}\big)|u| |\llcorner u \lrcorner_m|^\ka.
 \end{align}
 Using \cite[Lemma A.2]{brasco2}, we have
 \begin{align*}
 	A_p(u,g_\ka(u),\mb R^{2N})\geq \| G_\ka(u)\|^p_{W^{s_1,p}_0(\Om)} \quad\mbox{and } \ 
 	A_q(u,g_\ka(u),\mb R^{2N})\geq \| \tilde{G}_\ka(u)\|^q_{W^{s_2,q}_0(\Om)}\geq 0,
 \end{align*}
 where $\tilde{G}_\ka(t)=\int_{0}^{t}g^\prime_\ka(\tau)^{1/q}d\tau$. Therefore, from \eqref{eq21}, we obtain
 \begin{align*}
 	\| G_\ka(u)\|^p_{W^{s_1,p}_0(\Om)} \leq C \int_{\Om} \big(1+|u|^{p^*_{s_1}-1}\big)|u| |\llcorner u \lrcorner_m|^\ka.
 \end{align*}
 Then, rest of the proof follows exactly on the same lines of \cite[Theorem 3.3]{chenJFA}. \QED 
\end{proof}	
\begin{Remark}
 We remark that, as in \cite[Remark 3.4]{chenJFA}, the quantity $\|u \|_{L^\infty(\Om)}$ depends only on the constants $C_0$, $N$, $p$, $s_1$, $\|u\|_{W^{s_1,p}_0(\Om)}$ and the constant $M>0$ satisfying $\int_{\{|u|\geq M\}} |u|^{p^*_{s_1}}<\e$, for given $\e\in (0,1)$.
\end{Remark}

 \subsection{Interior regularity}
 We have the following interior regularity result.
\begin{Theorem}\cite[Theorem 2.10]{DDS}
  Let $2\le q\le p<\infty$. There exist $\al_0\in(0,1)$ and $C>0$ such that if $u\in\widetilde{W}^{s_1,p}(B_{R_0})\cap\widetilde{W}^{s_2,q}(B_{R_0})$ is bounded in $\mb R^N$ and satisfies $|(-\Delta)_p^{s_1}u+(-\Delta)_q^{s_2}u|\le K$ weakly in $B_{R_0}$, for some $K$ and $R_0>0$, then for all $r\in(0, R_0)$, 
 	\[  [u]_{C^{\al_0}(B_r)} \le \frac{C}{R_0^{\al_0}}\Big((KR_0^{ps_1})^\frac{1}{p-1} +\|u\|_{L^\infty(\mb R^N)} +T_{p-1}(u;R_0)+R_0^\frac{ps_1-qs_2}{p-1} \Big).\]
\end{Theorem}
 \subsubsection{Improved interior regularity}
 We first fix some notation which will be used in this subsection.
 For a measurable function $u:\mb R^N\to \mb R$ and $h\in\mb R^N$, we define 
    \begin{align*}
     u_h(x)=u(x+h), \quad \de_h u(x)=u_h(x)-u(x), \quad \de^2_h u(x)= \de_h(\de_h u(x))= u_{2h}(x)+u(x)-u_h(x).
    \end{align*} 
 Now, we recall the discrete Leibniz rule 
  \[ \de_h (uv)=u_h\de_h v + v \de_h u \quad\mbox{and } \de^2_h(uv)=u_{2h}\de^2_h v+2\de_h v \de_h u_h + v \de^2_h u. \]
 For $1\le m<\infty$ and $u\in L^m(\mb R^N)$, we set
  \begin{align*}
  	&[u]_{\mc N^{\ba,m}_{\infty}(\mb R^N)}:= \sup_{|h|>0} \bigg\| \frac{\de_h u}{|h|^\ba}\bigg\|_{L^m(\mb R^N)} \quad\mbox{for }  0<\ba\leq 1, \\
  	&[u]_{\mc B^{\ba,m}_{\infty}(\mb R^N)}:= \sup_{|h|>0} \bigg\| \frac{\de^2_h u}{|h|^\ba}\bigg\|_{L^m(\mb R^N)} \quad\mbox{for }  0<\ba< 2.
 \end{align*}

 \begin{Proposition}\label{propin1}
  Let $u\in W^{s_1,p}_{\rm loc}(B_{2R_0})\cap L^{p-1}_{s_1p}(\mb R^N) \cap L^{q-1}_{s_2q}(\mb R^N)$ be a local weak solution to the problem:
 	 \begin{align*}
 	 	(-\Delta)_p^{s_1}u+(-\Delta)_q^{s_2}u =f \quad\mbox{in }B_{2R_0},
 	 \end{align*}
  where $R_0\in (0,1)$ and $f\in L^\infty_{\mathrm{loc}}(B_{2R_0})$.
   Suppose that for some $m\geq p$ and $0<h_0<\frac{R_0}{10}$, we have 
   \begin{align*}
   	\sup_{0<|h|<h_0} \bigg\| \frac{\de^2_h u}{|h|^{s_1}}\bigg\|_{L^m(B_{R_0})} <\infty.
   \end{align*}
  Then, for every $4h_0<R\leq R_0 - 5h_0$, we have 
    \begin{align*}
     \sup_{0<|h|<h_0} \bigg\| \frac{\de^2_h u}{|h|^{s_1}}\bigg\|^{m+1}_{L^{m+1}(B_{R-4h_0})} 
     \leq C K_2(u,m) \left( \sup_{0<|h|<h_0} \bigg\| \frac{\de^2_h u}{|h|^{s_1}}\bigg\|^m_{L^m(B_{R+4h_0})} +1 \right),
    \end{align*}
  where $C= C(N,h_0,p,m,s_1)>0$ (which depends inversely on $h_0$) and $K_2(u,m)$ is given by \eqref{eqK(u)} (defined below). 
 \end{Proposition}
 \begin{proof}
  By the local boundedness property  (see Proposition \ref{localbdd} and Corollary \ref{corlocalbdd}), we have $u\in L^\infty_{\mathrm{loc}}(B_{2R_0})$.  Set the following: 
    \begin{align*}
    	r=R-4h_0 \quad\mbox{and } d\mu_{l,s}(x,y)=\frac{dxdy}{|x-y|^{N+ls}}, \quad\mbox{with }(l,s)\in \{ (p,s_1), (q,s_2)\}.
    \end{align*} 
  Fix $d\mu_1=d\mu_{(p,s_1)}$ and $d\mu_2=d\mu_{(q,s_2)}$. We take a test function $\phi\in W^{s_1,p}(B_R)\cap\ W^{s_2,q}(B_R)$ such that $\mathrm{supp}(\phi)\subset B_{(R+r)/2}$. Then, testing the equation with $\phi$ and $\phi_{-h}$, for $0<|h|<h_0$, and changing the variables, we obtain 
    \begin{align}\label{eqi1}
      &\int_{\mb R^{2N}} ([u_h(x)-u_h(y)]^{p-1}-[u(x)-u(y)]^{p-1})(\phi(x)-\phi(y))d\mu_1 \nonumber\\ 
      &+\int_{\mb R^{2N}} ([u_h(x)-u_h(y)]^{q-1}-[u(x)-u(y)]^{q-1})(\phi(x)-\phi(y))d\mu_2 \nonumber\\
      &=\int_{B_{2R_0}}\big( f(x+h)-f(x) \big)\phi(x)dx.
    \end{align}
 \end{proof}
 For $\ba\geq 1$ and $\nu>0$, we take 
  \begin{align*}
  	\phi:= \bigg[\frac{u_h-u}{|h|^\nu} \bigg]^\ba \eta^p=\bigg[\frac{\de_h u}{|h|^\nu}\bigg]^\ba \eta^p, \quad\mbox{for } 0<|h|<h_0,
  \end{align*}
 where $\eta\in C_c^2(\mb R^N)$ such that 
  \begin{align*}
  	\eta=1 \quad\mbox{in }B_r, \ \eta\equiv 0  \quad\mbox{in }B^c_{(R+r)/2}, \ \ 0\leq \eta \leq 1, \quad\mbox{and } |\na\eta|\leq\frac{C}{R-r}=C(4h_0)^{-1}.
  \end{align*}
 Now, dividing \eqref{eqi1} by $|h|>0$, for the above choice of $\phi$, we have 
  \begin{align*}
 	&\int_{\mb R^{2N}} \frac{([u_h(x)-u_h(y)]^{p-1}-[u(x)-u(y)]^{p-1})}{|h|^{1+\nu\ba}}\big([\de_hu(x)]^\ba \eta^p(x)-[\de_hu(y)]^\ba \eta^p(y) \big)d\mu_1 \nonumber\\ 
 	&+\int_{\mb R^{2N}} \frac{([u_h(x)-u_h(y)]^{q-1}-[u(x)-u(y)]^{q-1})}{|h|^{1+\nu\ba}}\big([\de_hu(x)]^\ba \eta^p(x)-[\de_hu(y)]^\ba \eta^p(y) \big)d\mu_2 \nonumber\\
 	&=\int_{B_{2R_0}}\frac{f(x+h)-f(x)}{|h|^{1+\nu\ba}} [\de_hu(x)]^\ba \eta^p(x)dx.
 \end{align*} 
 For $(l,s)\in \{(p,s_1),(q,s_2)\}$, we set the following:
  {\small \begin{align*}
  	&I_1(l,s)=\int_{B_R}\int_{B_R} \frac{([u_h(x)-u_h(y)]^{l-1}-[u(x)-u(y)]^{l-1})}{|h|^{1+\nu\ba}} \big([\de_hu(x)]^\ba \eta^p(x)-[\de_hu(y)]^\ba \eta^p(y) \big)d\mu, \\
  	&I_2(l,s)= \int_{B_{(R+r)/2}}\int_{\mb R^N\setminus B_R} \frac{([u_h(x)-u_h(y)]^{l-1}-[u(x)-u(y)]^{l-1})}{|h|^{1+\nu\ba}}[\de_hu(x)]^\ba \eta^p(x)d\mu, \\
  	&I_3(l,s)= -\int_{\mb R^N\setminus B_R}\int_{B_{(R+r)/2}} \frac{([u_h(x)-u_h(y)]^{l-1}-[u(x)-u(y)]^{l-1})}{|h|^{1+\nu\ba}}[\de_hu(y)]^\ba \eta^p(y)d\mu, \\
  	&I_4(f)=\int_{B_{R}}\frac{f(x+h)-f(x)}{|h|^{1+\nu\ba}} [\de_hu(x)]^\ba \eta^p(x)dx.
  \end{align*} }
 Therefore, we have 
  \begin{align}\label{eqi4}
  	I_1(p,s_1) \leq -I_1(q,s_2)+|I_2(p,s_1)|+|I_2(q,s_2)|+|I_3(p,s_1)|+|I_3(q,s_2)|+|I_4(f)|.
  \end{align}
 \textit{Step I}: Estimate of $I_1$.\\
 We observe that 
  \begin{align*}
 	[\de_hu(x)]^\ba \eta^p(x)-[\de_hu(y)]^\ba \eta^p(y) &= \big([\de_hu(x)]^\ba -[\de_hu(y)]^\ba \big) \frac{\eta^p(x)+\eta^p(y)}{2} \\
 	&\quad+ \big([\de_hu(x)]^\ba + [\de_hu(y)]^\ba \big) \frac{\eta^p(x)-\eta^p(y)}{2}.
 \end{align*} 
 Thus, 
 \begin{align*}
 	&([u_h(x)-u_h(y)]^{l-1}-[u(x)-u(y)]^{l-1}) \big([\de_hu(x)]^\ba \eta^p(x)-[\de_hu(y)]^\ba \eta^p(y) \big) \\
 	&\geq ([u_h(x)-u_h(y)]^{l-1}-[u(x)-u(y)]^{l-1}) \big([\de_hu(x)]^\ba -[\de_hu(y)]^\ba  \big) \frac{\eta^p(x)+\eta^p(y)}{2} \\
 	&\quad- |[u_h(x)-u_h(y)]^{l-1}-[u(x)-u(y)]^{l-1} | (|\de_hu(x)|^\ba +|\de_hu(y)|^\ba )\bigg|\frac{\eta^p(x)-\eta^p(y)}{2}\bigg|.
 \end{align*}
 Moreover, proceeding similarly to \cite[Proof of Proposition 4.1, pp. 814-815]{brascoH}, for $\e>0$, we obtain
{\small \begin{align*}
  &\big|[u_h(x)-u_h(y)]^{l-1}-[u(x)-u(y)]^{l-1} \big| (|\de_hu(x)|^\ba +|\de_hu(y)|^\ba )\bigg|\frac{\eta^p(x)-\eta^p(y)}{2}\bigg| \\
  &\leq \frac{2(l-1)}{l} \big( |u_h(x)-u_h(y)|^\frac{l-2}{2}+|u(x)-u(y)|^\frac{l-2}{2}  \big)
   \bigg| [u_h(x)-u_h(y)]^\frac{l}{2}-[u(x)-u(y)]^\frac{l}{2} \bigg| \\
   &\quad\times \big( |\de_hu(x)|^\ba+|\de_hu(y)|^\ba \big) \frac{\eta^\frac{p}{2}(x)+\eta^\frac{p}{2}(y)}{2} \big|\eta^\frac{p}{2}(x)-\eta^\frac{p}{2}(y)\big| \\
   & \leq \frac{C}{\e} \big( |u_h(x)-u_h(y)|^\frac{l-2}{2}+|u(x)-u(y)|^\frac{l-2}{2} \big)^2 ( |\de_hu(x)|^{\ba+1}+|\de_hu(y)|^{\ba+1} ) \big|\eta^\frac{p}{2}(x)-\eta^\frac{p}{2}(y)\big|^2 \\
   &+ C\e \big(|[u_h(x)-u_h(y)]^{l-1}-[u(x)-u(y)]^{l-1} |\big)([\de_hu(x)]^\ba-[\de_hu(y)]^\ba ) (\eta^p(x)+\eta^p(y)).
 \end{align*}}
 Therefore, for sufficiently small $\e>0$, using \cite[(A.5)]{brascoH}, we have
 \begin{align}\label{eqi2}
 	I_1 &\geq \frac{1}{C} \int_{B_R}\int_{B_R} \left[ \frac{([u_h(x)-u_h(y)]^{l-1}-[u(x)-u(y)]^{l-1})}{|h|^{1+\nu\ba}}([\de_hu(x)]^\ba -[\de_hu(y)]^\ba ) \right. \nonumber\\
 	& \left. \qquad\qquad\qquad\times (\eta^p(x)+\eta^p(y))d\mu \right] \nonumber\\
 	&-C \int_{B_R}\int_{B_R}\left[ \frac{|\de_hu(x)|^{\ba+1}+|\de_hu(y)|^{\ba+1}}{|h|^{1+\nu\ba}} \big[ |u_h(x)-u_h(y)|^\frac{l-2}{2}+|u(x)-u(y)|^\frac{l-2}{2} \big]^2 \right. \nonumber\\
 	& \left. \qquad\qquad\qquad\times \big|\eta^\frac{p}{2}(x)-\eta^\frac{p}{2}(y)\big|^2 \right] d\mu \nonumber\\
 	&\geq c \int_{B_R}\int_{B_R} \bigg| \frac{[\de_hu(x)]^\frac{\ba+l-1}{l}}{|h|^\frac{1+\nu\ba}{l}} - \frac{[\de_hu(y)]^\frac{\ba+l-1}{l}}{|h|^\frac{1+\nu\ba}{l}}\bigg|^l (\eta^p(x)+\eta^p(y))d\mu \nonumber\\
 	&- C \int_{B_R}\int_{B_R} \left[ \frac{|\de_hu(x)|^{\ba+1}+|\de_hu(y)|^{\ba+1}}{|h|^{1+\nu\ba}}\big( |u_h(x)-u_h(y)|^\frac{l-2}{2}+|u(x)-u(y)|^\frac{l-2}{2} \big)^2 \right. \nonumber\\
 	& \left. \qquad\qquad\qquad \times \big|\eta^\frac{p}{2}(x)-\eta^\frac{p}{2}(y)\big|^2 \right]d\mu. 
 \end{align}
 From the calculations of \cite[pp. 815-816]{brascoH}, we have
 \begin{align}\label{eqi3}
 	I_1(p)&\geq c \bigg[\frac{[\de_hu ]^\frac{\ba+p-1}{p}\eta}{|h|^\frac{1+\nu\ba}{p}}\bigg]^p_{W^{s_1,p}(B_R)}
 	- C \int_{B_R}\int_{B_R} \left[\frac{|\de_hu(x)|^{\ba+p-1}+|\de_hu(y)|^{\ba+p-1}}{|h|^{1+\nu\ba}} \right]|\eta(x)-\eta(y)|^p \nonumber\\
 	&\quad-C\int_{B_R} \int_{B_R} \left[ \frac{|\de_hu(x)|^{\ba+1}+|\de_hu(y)|^{\ba+1}}{|h|^{1+\nu\ba}}\big( |u_h(x)-u_h(y)|^\frac{p-2}{2}+|u(x)-u(y)|^\frac{p-2}{2} \big)^2 \right. \nonumber\\
 	&\qquad\qquad\qquad\left. \times \big|\eta^\frac{p}{2}(x)-\eta^\frac{p}{2}(y)\big|^2 \right]d\mu_1. 
 \end{align}
 Therefore, by observing that the first term on the right side of \eqref{eqi2} is non-negative and using \eqref{eqi3} in \eqref{eqi4}, we deduce that 
 \begin{align}\label{eqi5}
 	\bigg[\frac{[\de_hu ]^\frac{\ba+p-1}{p}\eta}{|h|^\frac{1+\nu\ba}{p}}\bigg]^p_{W^{s_1,p}(B_R)} \leq C \Big[I_{11}(p)+I_{12}(p)+I_{11}(q)+\sum_{l\in\{p,q\}} (|I_2(l)|+|I_3(l)|) +|I_4(f)|\Big],
 \end{align}
 where 
 \begin{align*}
  &I_{11}(l)= \int_{B_R} \int_{B_R} \left[ \frac{|\de_hu(x)|^{\ba+1}+|\de_hu(y)|^{\ba+1}}{|h|^{1+\nu\ba}}\big( |u_h(x)-u_h(y)|^\frac{l-2}{2}+|u(x)-u(y)|^\frac{l-2}{2} \big)^2 \right. \\ &\left.\qquad\qquad\qquad\times \big|\eta^\frac{p}{2}(x)-\eta^\frac{p}{2}(y)\big|^2 \right]d\mu\\
  &I_{12}(p)= \int_{B_R}\int_{B_R} \left( \frac{|\de_hu(x)|^{\ba+p-1}+|\de_hu(y)|^{\ba+p-1}}{|h|^{1+\nu\ba}} \right)|\eta(x)-\eta(y)|^p d\mu_1.
 \end{align*}
 \textit{Step II}: Estimates of $I_{11}$ and $I_{12}$.\\
 Using the bounds on $\eta$ and $|\na\eta|$, we deduce that
  \begin{align*}
  	\tl I_{11}:&=\int_{B_R} \int_{B_R} \frac{|\de_hu(x)|^{\ba+1}}{|h|^{1+\nu\ba}} |u(x)-u(y)|^{l-2} \big|\eta^\frac{p}{2}(x)-\eta^\frac{p}{2}(y)\big|^2 d\mu \nonumber\\
  	&\leq \frac{C}{h_0^2}\int_{B_R} \int_{B_R} \frac{|u(x)-u(y)|^{l-2}}{|x-y|^{N+sl-2}} \frac{|\de_hu(x)|^{\ba+1}}{|h|^{1+\nu\ba}} dxdy.
  \end{align*}
 For $l=2$, noticing the fact that $R<1$, we have 
  \begin{align}\label{eqi6}
  	\tl I_{11}\leq \frac{C(N,s)}{h_0^2} \| u \|_{L^\infty(B_{R+h_0})} \int_{B_R} \frac{|\de_hu(x)|^{\ba}}{|h|^{1+\nu\ba}}dx.
  \end{align}
 For $l>2$, we take $\e<\min\{\frac{q-2}{2},\frac{1}{s_1}-1\}>0$. Using Young's inequality (with exponents $m/(l-2)$ and $m/(m-l+2)$), we get
 \begin{align}\label{eqi7}
 	\tl I_{11}&\leq \frac{C}{h_0^2}\int_{B_R} \int_{B_R} \frac{|u(x)-u(y)|^{l-2}}{|x-y|^{N+sl-2}} \frac{|\de_hu(x)|^{\ba+1}}{|h|^{1+\nu\ba}} dxdy \nonumber\\
 	&\leq \frac{C}{h_0^2}\int_{B_R} \int_{B_R} \frac{|u(x)-u(y)|^{m}}{|x-y|^{N+ms\frac{(l-2-\e)}{l-2}}} + \left(\frac{C}{h_0}\right)^\frac{m}{m-l+2}\int_{B_R} \int_{B_R} |x-y|^{\frac{m(2-2s-\e s)}{m-l+2}-N} \frac{|\de_hu(x)|^{\frac{(\ba+1)m}{m-l+2}}}{|h|^\frac{(1+\nu\ba)m}{m-l+2}} \nonumber\\
 	&\leq C [u]^m_{W^{\frac{s(l-2-\e)}{l-2},m}(B_{R+h_0})} + C \| u \|^{\frac{m}{m-l+2}}_{L^\infty(B_{R+h_0})} \int_{B_R}  \frac{|\de_hu(x)|^{\frac{\ba m}{m-l+2}}}{|h|^\frac{(1+\nu\ba)m}{m-l+2}}dx, 
  \end{align}
 where we have used the fact $\frac{m(2-2s-\e s)}{m-l+2}>0$ and $C=C(N,s,l,h_0)>0$ is a constant (which depends inversely on $h_0$).  
 For the first term, using the Sobolev embedding result \cite[Lemma 2.6]{brascoH}, we have
  \begin{align}\label{eqi8}
  	[u]^m_{W^{\frac{s(l-2-\e)}{l-2},m}(B_{R+h_0})}\leq C \left[\sup_{0<|h|<h_0}\bigg\| \frac{\de^2_h u}{|h|^{s_1}}\bigg\|^m_{L^m(B_{R+4h_0})} + \| u \|^m_{L^\infty(B_{R+h_0})} \right].
  \end{align}
 Combining  \eqref{eqi7}, \eqref{eqi8} and \eqref{eqi6} together with Young's inequality, we obtain 
 \begin{align}\label{eqi9}
 	|I_{11}| \leq C K_1(u)\left[ \int_{B_R}  \frac{|\de_hu(x)|^{\frac{\ba m}{m-l+2}}}{|h|^\frac{(1+\nu\ba)m}{m-l+2}}dx + \sup_{0<|h|<h_0}\bigg\| \frac{\de^2_h u}{|h|^{s_1}}\bigg\|^m_{L^m(B_{R+4h_0})} +1 \right],
 \end{align}
 where $C=C(N,l,s,h_0)>0$ is a constant (which depends inversely on $h_0$) and 
 \[ K_1(u)=K_1(u,m,l,R_0)=\max\{ \| u \|^{\frac{m}{m-l+2}}_{L^\infty(B_{R_0})}, \| u \|^m_{L^\infty(B_{R_0})}, \| u \|_{L^\infty(B_{R_0})} \}>0. \]
For $I_{12}(p)$, we use the estimate of \cite[p. 819]{brascoH},
 \begin{align}\label{eqi10}
 	I_{12}=2\int_{B_R\times B_R}  \frac{|\de_hu(x)|^{\ba+p-1}}{|h|^{1+\nu\ba}} {|\eta(x)-\eta(y)|^p} d\mu \leq C  \left[ \int_{B_R}  \frac{|\de_hu(x)|^{\frac{\ba m}{m-l+2}}}{|h|^\frac{(1+\nu\ba)m}{m-l+2}}dx +\| u \|^{\frac{m(p-1)}{p-2}}_{L^\infty(B_{R+h_0})} \right], 
 \end{align}
 where $C>0$ is a constant appeared as before.
 
 \textit{Step III}: Estimates of $I_2$, $I_3$ and $I_4$.\\
 To estimate $I_2$ and $I_3$, we observe that
  \begin{align}\label{eqi11}
  	\big|& ([u_h(x)-u_h(y)]^{l-1}-[u(x)-u(y)]^{l-1})[\de_hu(x)]^\ba \big| \nonumber\\
  	 &\qquad\leq C(l)\big( \|u\|^{l-1}_{L^\infty(B_{R_0})}+ |u_h(y)|^{l-1}+|u(y)|^{l-1} \big)|\de_hu(x)|^\ba.
  \end{align}
 For the first term, by observing $B_{(R-r)/2}(x)\subset B_R$, whenever $x\in B_{(r+R)/2}$, 
 \begin{align}\label{eqi12}
 	\int_{\mb R^N\setminus B_R} \frac{dy}{|x-y|^{N+sl}}\leq \int_{\mb R^N\setminus B_\frac{R-r}{2}(x)} \frac{dy}{|x-y|^{N+sl}}\leq C(N,s,l,h_0),
 \end{align}
 where we have used $R-r=4h_0$, and $C(N,s,l,h_0)>0$ depends inversely on $h_0$. For the other terms, we use \cite[Lemmas 2.2, 2.3]{brascoH}, and the fact that $4h_0=R-r<R \leq R_0<1$. Thus
 \begin{align}\label{eqi13}
  	\int_{\mb R^N\setminus B_R} \frac{|u(y)|^{l-1}}{|x-y|^{N+sl}}dy
  	&\leq \frac{1}{(R-r)^{N+sl}} \left[ R^{N+sl}\int_{\mb R^N\setminus B_{R_0}}\frac{|u(y)|^{l-1}}{|y|^{N+sl}}dy+ \int_{B_{R_0}}|u(y)|^{l-1}dy \right] \nonumber\\
  	&\leq C(N,l,s,h_0) \big( \|u\|^{l-1}_{L^\infty(B_{R_0})} + T_{l-1}(u;R_0)^{l-1} \big),
 \end{align} 
 where $T_{l-1}$ is the tail as defined in definition \ref{defTail}.
 The estimate of $I_4$ is straightforward. Indeed, by noticing the fact that $f\in L^\infty_{\mathrm{loc}}(B_{2R_0})$ and $B_{R+h_0}\Subset B_{2R_0}$, we obtain
  \begin{align}\label{eqif}
  	|I_4(f)|\leq 2\|f\|_{L^\infty(B_{R+h_0})}\int_{B_{R}}\frac{|\de_hu(x)|^\ba \eta^p(x)}{|h|^{1+\nu\ba}} dx \leq 2\|f\|_{L^\infty(B_{R+h_0})}\int_{B_{(R+r)/2}}\frac{|\de_hu(x)|^\ba }{|h|^{1+\nu\ba}} dx. 
  \end{align} Therefore, combining \eqref{eqi11}, \eqref{eqi12}, \eqref{eqi13}, \eqref{eqif}, and then using Young's inequality, we deduce that
 \begin{align}\label{eqi14}
 	|I_2|+|I_3|+|I_4|&\leq C \big( \|u\|^{l-1}_{L^\infty(B_{R_0})} + T_{l-1}(u;R_0)^{l-1}+\|f\|_{L^\infty(B_{R+h_0})} \big) \int_{B_{(R+r)/2}}\frac{|\de_h u(x)|^\ba}{|h|^{1+\nu\ba}}dx \nonumber\\
 	&\leq C \big( \|u\|^{l-1}_{L^\infty(B_{R_0})} + T_{l-1}(u;R_0)^{l-1}+\|f\|_{L^\infty(B_{R_0})} \big) \left[1+ \int_{B_R}  \frac{|\de_hu(x)|^{\frac{\ba m}{m-l+2}}}{|h|^\frac{(1+\nu\ba)m}{m-l+2}}dx \right].
 \end{align}
\textit{Step IV}: Conclusion.\\
 Using the estimates \eqref{eqi9}, \eqref{eqi10} and \eqref{eqi14} in \eqref{eqi5}, for $l\in \{p,q\}$, we obtain  
 {\small\begin{align}\label{eqi15a}
 	\bigg[\frac{[\de_hu ]^\frac{\ba+p-1}{p}\eta}{|h|^\frac{1+\nu\ba}{p}}\bigg]^p_{W^{s_1,p}(B_R)} 
 	\leq CK_2(u,m) \left[ \sum_{l}\int_{B_R} \frac{|\de_hu(x)|^{\frac{\ba m}{m-l+2}}}{|h|^\frac{(1+\nu\ba)m}{m-l+2}}dx + \sup_{0<|h|<h_0}\bigg\| \frac{\de^2_h u}{|h|^{s_1}}\bigg\|^m_{L^m(B_{R+4h_0})}+1 \right],
 \end{align}}
 where $C= C(N,l,s,h_0)>0$ is a constant as appeared before, and \begin{align}\label{eqK(u)} 
 	K_2(u,m,R_0):=1+T_{p-1}(u;R_0)^{p-1}+T_{q-1}(u;R_0)^{q-1}+ \| u \|^{\frac{m(p-1)}{p-2}}_{L^\infty(B_{R_0})}+\|f\|_{L^\infty(B_{R_0})} >0.
 \end{align}
 Note that we have suppressed the term $R_0$ when it is not important (observe that $K_2\geq K_1$). For the term corresponding to $l=q$ in \eqref{eqi15a}, applying Young's inequality (with exponents $\frac{m-q+2}{m-p+2}$ and $\frac{m-q+2}{p-q}$), we get
 {\small\begin{align}\label{eqi15}
 	\bigg[\frac{[\de_hu ]^\frac{\ba+p-1}{p}\eta}{|h|^\frac{1+\nu\ba}{p}}\bigg]^p_{W^{s_1,p}(B_R)} 
 	\leq CK_2(u,m) \left[ \int_{B_R} \frac{|\de_hu(x)|^{\frac{\ba m}{m-p+2}}}{|h|^\frac{(1+\nu\ba)m}{m-p+2}}dx + \sup_{0<|h|<h_0}\bigg\| \frac{\de^2_h u}{|h|^{s_1}}\bigg\|^m_{L^m(B_{R+4h_0})}+1 \right].
 \end{align}}
 Let $\xi,h\in\mb R^N$ such that $0<|\xi|,|h|<h_0$. Then, from \cite[(4.11), p. 821]{brascoH} and using the discrete Leibniz rule, we have
 \begin{align}\label{eqi16}
 	\int_{B_r} \frac{|\de_\xi \de_h u|^{\ba-1+p}}{|\xi|^{s_1p}|h|^{1+\nu\ba}} 
 	&\leq C \bigg\| \eta\frac{\de_\xi}{|\xi|^{s_1}}\left( \frac{[\de_h u]^\frac{\ba-1+p}{p}}{|h|^\frac{1+\nu\ba}{p}} \right) \bigg\|^p_{L^p(\mb R^N)} \nonumber\\
 	&\leq C \bigg\| \frac{\de_\xi}{|\xi|^{s_1}}\left( \frac{[\de_h u]^\frac{\ba-1+p}{p}\eta}{|h|^\frac{1+\nu\ba}{p}} \right) \bigg\|^p_{L^p(\mb R^N)} + C \bigg\| \frac{\de_\xi\eta}{|\xi|^{s_1}} \frac{\big([\de_h u]^\frac{\ba-1+p}{p}\big)_\xi}{|h|^\frac{1+\nu\ba}{p}}  \bigg\|^p_{L^p(\mb R^N)},
 \end{align}
where $C=C(p,\ba)>0$ is a constant. To estimate the first term, we use \cite[Proposition 2.6]{brascoHS} (with $r=(R+r)/2$ and $R=R$ there). Therefore,
\begin{align}\label{eqi17}
	\sup_{\xi>0} \bigg\| \frac{\de_\xi}{|\xi|^{s_1}}\left( \frac{[\de_h u]^\frac{\ba-1+p}{p}\eta}{|h|^\frac{1+\nu\ba}{p}} \right) \bigg\|^p_{L^p(\mb R^N)} 
	\leq C(N,h_0,s_1,p) \bigg[\frac{[\de_hu ]^\frac{\ba+p-1}{p}\eta}{|h|^\frac{1+\nu\ba}{p}}\bigg]^p_{W^{s_1,p}(B_R)}.
\end{align} 
 For the second term in \eqref{eqi16}, we proceed similar to \cite[(4.14), p. 821]{brascoH},
 \begin{align}\label{eqi18}
 	\bigg\| \frac{\de_\xi\eta}{|\xi|^{s_1}} \frac{\big([\de_h u]^\frac{\ba-1+p}{p}\big)_\xi}{|h|^\frac{1+\nu\ba}{p}}  \bigg\|^p_{L^p(\mb R^N)}
 	&\leq C \| u\|^{p-1}_{L^\infty(B_{R_0})} \int_{B_R} \frac{|\de_h u|^\ba}{|h|^{1+\nu\ba}}dx \nonumber\\
 	&\leq C K_2(u,m) \left[ 1+ \int_{B_R} \frac{|\de_hu(x)|^{\frac{\ba m}{m-p+2}}}{|h|^\frac{(1+\nu\ba)m}{m-p+2}}dx \right].
 \end{align}
 From \eqref{eqi16}, \eqref{eqi17} and \eqref{eqi18} with $\xi=h$, and then using \eqref{eqi15}, we obtain 
 {\small\begin{align*}
 	\sup_{0<|h|<h_0}\int_{B_r} \frac{| \de^2_h u|^{\ba-1+p}}{|h|^{1+s_1p+\nu\ba}}dx
 	\leq C K_2(u,m) \sup_{0<|h|<h_0} \left[ \int_{B_R} \frac{|\de_hu(x)|^{\frac{\ba m}{m-p+2}}}{|h|^\frac{(1+\nu\ba)m}{m-p+2}}dx + \bigg\| \frac{\de^2_h u}{|h|^{s_1}}\bigg\|^m_{L^m(B_{R+4h_0})}+1  \right],
 \end{align*}}
 where the quantities $C=C(N,h_0,p,m,s_1,\ba)>0$ and $K_2(u,m)$ are as appeared before. Using \cite[Lemma 2.6]{brascoH}, for the first term on the above expression, we get
 {\small\begin{align}\label{eqi19}
 	\sup_{0<|h|<h_0}\int_{B_r} \frac{| \de^2_h u|^{\ba-1+p}}{|h|^{1+s_1p+\nu\ba}}dx
 	\leq C K_2(u,m) \sup_{0<|h|<h_0} \left[   \bigg\| \frac{\de^2_h u}{|h|^{\nu+\frac{1}{\ba}}}\bigg\|^{\frac{\ba m}{m-p+2}}_{L^{\frac{\ba m}{m-p+2}}(B_{R+4h_0})}  + \bigg\| \frac{\de^2_h u}{|h|^{s_1}}\bigg\|^m_{L^m(B_{R+4h_0})}+1  \right],
 \end{align}}
where we are using $(1+\nu\ba)/\ba<1$. Next, we specify the symbols $\ba$ and $\nu$ as below:
 \begin{align*}
 	\ba= m-p+2, \quad \nu= \frac{(m-p+2)s_1-1}{m-p+2}.
 \end{align*}
Then, it is easy to observe that
 \begin{align*}
 	\frac{1+s_1p+\nu\ba}{\ba-1+p}=s_1+\frac{s_1}{m+1}, \ \ \ba-1+p=m+1, \ \ \frac{m\ba}{m-p+2}=m, \ \ \frac{1+\nu\ba}{\ba}=s_1.
 \end{align*} 
 Therefore, on account of \eqref{eqi19}, we have
 \begin{align*}
 	\sup_{0<|h|<h_0}  \bigg\| \frac{\de^2_h u}{|h|^{s_1+\frac{s_1}{m+1}}} \bigg\|^{m+1}_{L^{m+1}(B_{r})} 
 	\leq C K_2(u,m) \sup_{0<|h|<h_0} \left[   \bigg\| \frac{\de^2_h u}{|h|^{s_1}}\bigg\|^m_{L^m(B_{R+4h_0})}+1  \right].
 \end{align*}
 By noting $r=R-4h_0$ and $|h|<h_0$, the above expression implies that (up to modifying the constant $C$)
 \begin{align*}
 	\sup_{0<|h|<h_0}  \bigg\| \frac{\de^2_h u}{|h|^{s_1}} \bigg\|^{m+1}_{L^{m+1}(B_{R-4h_0})} 
 	\leq C K_2(u,m) \sup_{0<|h|<h_0} \left[   \bigg\| \frac{\de^2_h u}{|h|^{s_1}}\bigg\|^m_{L^m(B_{R+4h_0})}+1  \right],
 \end{align*}
 where $C= C(N,h_0,p,m,s_1)>0$ depends inversely on $h_0$ and $K_2(u,m)$ is given by \eqref{eqK(u)}. \QED
 \begin{Corollary}\label{corimpint}
 Let $u$, $f$, $R_0$ and $h_0$ be as in proposition \ref{propin1}. Let $\nu\in (0,1)$ and $\ba>1$ are such that $(1+\nu\ba)/\ba<1$. Further assume that $u\in C^{0,s_1-\e}_{\rm loc}(\Om)$, for all $\e\in (0,s_1)$, and 
 \begin{align*}
 	\sup_{0<|h|<h_0} \bigg\| \frac{\de^2_h u}{|h|^{(1+\nu\ba)/\ba}}\bigg\|_{L^\ba(B_{R_0})} <\infty.
 \end{align*}
 Then, for every $4h_0<R\leq R_0 - 5h_0$, we have 
 \begin{align*}
 	\sup_{0<|h|<h_0} \bigg\| \frac{\de^2_h u}{|h|^\frac{1+s_1p+\nu\ba}{\ba+p-1}}\bigg\|^{\ba+p-1}_{L^{\ba+p-1}(B_{R-4h_0})} 
 	\leq C K_2(u,m)^{p-1} \left( \sup_{0<|h|<h_0} \bigg\| \frac{\de^2_h u}{|h|^\frac{1+\nu\ba}{\ba}}\bigg\|^\ba_{L^\ba(B_{R+4h_0})} +1 \right),
 \end{align*}
 where $C= C(N,h_0,p,m,s_1)>0$ (which depends inversely on $h_0$) and $K_2(u,m)$ is given by \eqref{eqK(u)}. 
 \end{Corollary}
\begin{proof}
 Fix $\e\in (0,\min\{ s_1, s_1+\frac{2-s_1p}{p-2}\})$. Thus, for all  $R\in (0,1)$ and $x_0\in\mb R^N$ such that $B_R(x_0)\Subset\Om$, we have 
 \begin{align}\label{eqi32}
 	[u]_{C^{s_1-\e}(B_R(x_0))}\leq CK_2(u).
 \end{align}
  Using \eqref{eqi32}, and the bounds on $\eta$ and $|\na\eta|$, for $(l,s)\in\{(p,s_1), (q,s_2)\}$, we deduce that
 \begin{align*}
 	\tl I_{11}:&=\int_{B_R}  \frac{|u(x)-u(y)|^{l-2}}{|x-y|^{N+sl}} \big|\eta^\frac{p}{2}(x)-\eta^\frac{p}{2}(y)\big|^2 dy \nonumber\\
 	&\leq CK_2(u)^{l-2}\int_{B_R} |x-y|^{(l-2)(s_1-\e)+2-N-sl}dy<\infty,
 \end{align*} 
 where we have used the fact $(l-2)(s_1-\e)+2-sl>0$. Consequently, we have
 \begin{align*}
 	|I_{11}(l)|\leq C K_2(u)^{l-2} \| u \|_{L^\infty(B_{R+h_0})} \int_{B_R}\frac{|\de_h u(x)|^\ba}{|h|^{1+\nu\ba}}dx.
 \end{align*}
 On a similar note,
 \begin{align*}
 	|I_{12}|+|I_2|+|I_3|+|I_4|\leq C K_2(u)^{l-1} \int_{B_R}\frac{|\de_h u(x)|^\ba}{|h|^{1+\nu\ba}}dx.
 \end{align*}
 Coupling these with \eqref{eqi5} and using in \eqref{eqi17} (for $\xi=h$ there), we obtain
 \begin{align*}
 	\sup_{0<|h|<h_0} \int_{B_r}\frac{|\de^2_h u|^{\ba-1+p}}{|h|^{1+s_1p+\nu\ba}}dx
 	\leq C K_2(u,m)^{p-2} \sup_{0<|h|<h_0}\int_{B_R}\frac{|\de_h u(x)|^\ba}{|h|^{1+\nu\ba}}dx.
 \end{align*}
 Therefore, employing \cite[Lemma 2.6]{brascoH}, we get
 \begin{align*}
 	\sup_{0<|h|<h_0} \int_{B_{R-4h_0}}\frac{|\de^2_h u|^{\ba-1+p}}{|h|^{1+s_1p+\nu\ba}}dx
 	\leq C K_2(u,m)^{p-1} \left(\sup_{0<|h|<h_0} \int_{B_{R+4h_0}}\frac{|\de^2_h u(x)|^\ba}{|h|^{1+\nu\ba}}dx+1\right).
 \end{align*}
This proves the corollary. \QED
\end{proof}
 
 Now, we prove our interior H\"older regularity result.\\
 \textbf{Proof of Theorem \ref{impintreg}}: For simplicity in notation, we take $x_0=0$. By using the local boundedness property (Proposition \ref{localbdd} and Corollary \ref{corlocalbdd}), we have $u\in L^\infty_{\mathrm{loc}}(\Om)$. We proceed as below.\\
 \textbf{Step I}: We first prove $u\in C^{0,\sg}_{\rm loc}(\Om)$ for all $\sg\in (0,s_1)$.\\
  Since $\sg\in (0,s_1)$, there exists $i_\infty\in\mb N$ such that 
  \begin{align*}
 	s_1-\sg > \frac{N}{p+i_\infty}.
  \end{align*}
Set $h_0=\frac{\bar R_0}{64 i_\infty} $ and define the following sequences:
  \begin{align*}
	m_i=p+i, \quad\mbox{and } R_i= \frac{7\bar R_0}{8}-4(2i+1)h_0 \quad\mbox{for } i=0,\dots,i_\infty.
  \end{align*} 
  We then observe that 
  \[ R_0+4h_0=\frac{7\bar R_0}{8}, \quad R_i-4h_0=R_{i+1}+4h_0, \quad R_{i_\infty-1}-4h_0=\frac{3\bar R_0}{4} \quad\mbox{and }4h_0< R_i \leq \bar R_0-5h_0. \]
 Thus, by applying Proposition \ref{propin1} for the following choices:
 \[ R=R_i, \quad m=m_i=p+i \quad\mbox{for } i=0,\dots,i_\infty, \]
 we obtain
 {\small\begin{align*}
 	& \sup_{0<|h|<h_0}  \bigg\| \frac{\de^2_h u}{|h|^{s_1}} \bigg\|_{L^{m_{i+1}}(B_{R_{i+1}+4h_0})} 
 	\leq C K_2(u) \left[ \sup_{0<|h|<h_0}  \bigg\| \frac{\de^2_h u}{|h|^{s_1}}\bigg\|_{L^{m_i}(B_{ R_i+4h_0})}+1  \right],
 	\end{align*}} 
 for $i=0,\dots,i_\infty-2$, and
 {\small\begin{align*}
 	&\sup_{0<|h|<h_0}  \bigg\| \frac{\de^2_h u}{|h|^{s_1}} \bigg\|_{L^{m_{i_\infty}}(B_{(3\bar R_0)/4})} 
 	\leq C K_2(u) \left[ \sup_{0<|h|<h_0}  \bigg\| \frac{\de^2_h u}{|h|^{s_1}} \bigg\|_{L^{m_{i_\infty-1}}(B_{R_{i_\infty-1}+4h_0})} +1  \right].
 \end{align*}}
 We note that $\de_h u =\frac{\de_{2h}u-2\de_h u}{2}$. Therefore, applying \cite[Proposition 2.6]{brascoHS}, we obtain 
 \begin{align*}
 	\sup_{0<|h|<h_0}\bigg\| \frac{\de^2_h u}{|h|^{s_1}}\bigg\|_{L^{p}(B_{(7\bar R_0)/8})} 
 	&\leq C \sup_{0<|h|<2h_0}\bigg\| \frac{\de_h u}{|h|^{s_1}}\bigg\|_{L^{p}(B_{(7\bar R_0)/8})} \\
 	&\leq C \Big( [u]_{W^{s_1,p}(B_{\frac{7\bar R_0}{8}+2h_0})} + \| u \|_{L^{\infty}(B_{\frac{7\bar R_0}{8}+2h_0})}    \Big). 
 \end{align*}
From the definition of $K_2(u,m)$ (see \eqref{eqK(u)} with $R_0=\bar R_0$), it is clear that
 $$K_2(u,m)\leq 2K_2 (u,m_1) \quad \mbox{and } \| u \|_{L^\infty(B_{\bar R_0})}\leq K_2(u,m), \quad\mbox{for all }m_1\geq m.$$ 
  Thus, from the above, we conclude that 
 \begin{align*}
 	\sup_{0<|h|<h_0}\bigg\| \frac{\de^2_h u}{|h|^{s_1}}\bigg\|_{L^{p}(B_{(7\bar R_0)/8})} \leq C K_2(u,m)  \big( [u]_{W^{s_1,p}(B_{\bar R_0})} + 1 \big).
 \end{align*}
 Iterating the above process (the process is finite), we get
 \begin{align}\label{eqi21}
 	\sup_{0<|h|<h_0}\bigg\| \frac{\de^2_h u}{|h|^{s_1}}\bigg\|_{L^{m_{i_\infty}}(B_{(3\bar R_0)/4})} \leq \big(C K_2(u,m_{i_\infty})\big)^{i_{\infty}}  \big( [u]_{W^{s_1,p}(B_{\bar R_0})} + 1 \big).
 \end{align}
 We take $\psi\in C_c^\infty(B_{(5\bar R_0)/8})$ such that 
\[ 0\leq \psi \leq 1, \quad \psi=1 \mbox{ in } B_{\bar R_0/2}, \quad |\na\psi|, |\na^2\psi|\leq C. \]
This implies that 
\[ \frac{|\de_h \psi|}{|h|^{s_1}}, \frac{|\de^2_h \psi|}{|h|^{s_1}}\leq C \quad\mbox{for all }|h|>0. \]
Then, using the discrete Leibniz rule on $\de^2_h$, for $0<|h|<h_0$, we obtain
\begin{align}\label{eqi20}
	[u\psi]_{\mc B^{s_1,m_{i_\infty}}_{\infty}(\mb R^N)}&= \sup_{|h|>0} \bigg\| \frac{\de^2_h (u\psi)}{|h|^{s_1}}\bigg\|_{L^{m_{i_\infty}}(\mb R^N)} \nonumber\\
	&\leq C \sup_{|h|>0} \left[\bigg\| \frac{\psi_{2h}\de^2_h u}{|h|^{s_1}}\bigg\|_{L^{m_{i_\infty}}(\mb R^N)} + \bigg\| \frac{\de_h\psi\de_h u}{|h|^{s_1}}\bigg\|_{L^{m_{i_\infty}}(\mb R^N)} + \bigg\| \frac{u\de^2_h \psi}{|h|^{s_1}}\bigg\|_{L^{m_{i_\infty}}(\mb R^N)}  \right] \nonumber\\
	&\leq C  \left[ \sup_{0<|h|<h_0}\bigg\| \frac{\de^2_h u}{|h|^{s_1}}\bigg\|_{L^{m_{i_\infty}}(B_{(3\bar R_0)/4})} + \| u \|_{L^{m_{i_\infty}}(B_{(3\bar R_0)/4})} \right].
\end{align}
 Therefore, using \cite[Lemma 2.4]{brascoH} together with \eqref{eqi20} and \eqref{eqi21}, we obtain
  \begin{align*}
  	[u\psi]_{\mc N^{s_1,m_{i_\infty}}_{\infty}(\mb R^N)} \leq C(N,s_1,m) [u\psi]_{\mc B^{s_1,m_{i_\infty}}_{\infty}(\mb R^N)} \leq (C K_2(u,m_{i_\infty})\big)^{i_\infty}  \big( [u]_{W^{s_1,p}(B_{\bar R_0})} + 1 \big).
  \end{align*}
 Since $s_1 m_{i_\infty}>N$ and $\sg<s_1-N/m_{i_\infty}$, employing the embedding result of \cite[Theorem 2.8]{brascoH}, we get
  \begin{align*}
  	[u]_{C^\sg(B_r)} \leq [u]_{C^\sg(B_{\frac{\bar R_0}{2}})} 
  	&\leq C \big( [u\psi]_{\mc N^{s_1,m_{i_\infty}}_{\infty}(\mb R^N)} \big)^\frac{\sg m_{i_\infty}+N}{s_1 m_{i_\infty}} \big( \| u\psi \|_{L^{m_{i_\infty}}(\mb R^N)} \big)^\frac{(s_1-\sg) m_{i_\infty}-N}{s_1 m_{i_\infty}} \\ 
  	&\leq (C K_2(u,m_{i_\infty})\big)^{i_\infty} \big( [u]_{W^{s_1,p}(B_{\bar R_0})} + 1 \big)^{1+N/(s_1m_{i_\infty})} 
  	\\
  	&\leq (C K_2(u,m_{i_\infty})\big)^{i_\infty} \big( [u]_{W^{s_1,p}(B_{\bar R_0})} + 1 \big)^{2-\sg/s_1},
  \end{align*}
 where we have used the fact that $\| u \|_{L^\infty(B_{\bar R_0})}\leq K_2(u,m_{i_\infty})$ and $\sg m_{i_\infty}+N<s_1 m_{i_\infty}$. This completes the proof for the case $\sg\in (0,s_1)$. \\ 
 \textbf{Step II}: We prove $u\in C^{0,\sg}_{\rm loc}(\Om)$ for all $\sg\in (0,\min\{1,\frac{ps_1}{p-1}\})$.\\
We proceed similar to \cite[Theorem 5.2]{brascoH}, and define the following iterative sequences:
 \begin{align*}
 	\ba_0=p, \ \ba_{i+1}=\ba_i+p-1=p+i(p-1); \ \nu_0 = s_1-\frac{1}{p}, \ \nu_{i+1}=\frac{\nu_i\ba_i+s_1p}{\ba_{i+1}}.
 \end{align*}
 It is easy to observe that $\lim_{i\to\infty} \nu_i = \frac{s_1p}{p-1}$. Then, we have the following two cases.\\
 \textit{Case (i)}: If $s_1p\leq p-1$.\\
 Fix $\sg\in (0,\frac{s_1p}{p-1})$ and choose $i_\infty\in\mb N$ such that 
 $ \sg<\nu_{i_\infty}-\frac{N-1}{\ba_{i_\infty}}$.
 Set 
 \begin{align*}
 	h_0=\frac{\bar R_0}{64i_\infty}, \quad\mbox{and } R_i= \frac{7\bar R_0}{8}-4(2i+1)h_0 \quad\mbox{for } i=0,\dots,i_\infty.
 \end{align*} 
 From step I, we have $u\in C^{0,\sg}_{\rm loc}(\Om)$. Therefore by applying corollary \ref{corimpint}, for $\ba=\ba_i$, $\nu=\nu_i$ and $R=R_i$, we obtain
 {\small\begin{align*}
	 \sup_{0<|h|<h_0}  \bigg\| \frac{\de^2_h u}{|h|^{\nu_{i+1}+\frac{1}{\ba_{i+1}}}} \bigg\|_{L^{\ba_{i+1}}(B_{R_{i+1}+4h_0})} 
		\leq C K_2(u) \left[ \sup_{0<|h|<h_0}  \bigg\| \frac{\de^2_h u}{|h|^{\nu_i+\frac{1}{\ba_i}}}\bigg\|_{L^{\ba_i}(B_{ R_i+4h_0})}+1  \right], 
 \end{align*}}
 for $i=0,\dots,i_\infty-2$, and 
 {\small\begin{align*}
		\sup_{0<|h|<h_0}  \bigg\| \frac{\de^2_h u}{|h|^{\nu_{i_\infty}+\frac{1}{\ba_{i_\infty}}}} \bigg\|_{L^{\ba_{i_\infty}}(B_{(3\bar R_0)/4})} 
		\leq C K_2(u) \left[ \sup_{0<|h|<h_0}  \bigg\| \frac{\de^2_h u}{|h|^{\nu_{i_\infty-1}+\frac{1}{\ba_{i_\infty-1}}}} \bigg\|_{L^{\ba_{i_\infty-1}}(B_{R_{i_\infty-1}+4h_0})} +1  \right].
\end{align*}}
Now, proceeding similarly to the proof of theorem \ref{impintreg}, we obtain that $u\in C^{0,\sg}(B_{\bar R_0/2}(x_0))$.\\
\textit{Case (ii)}: If $s_1p>(p-1)$.\\
 The proof, in this case, runs analogously to \cite[Proof of Theorem 5.2, pp. 831-833]{brascoH}.  \QED

 \subsection{Boundary regularity and maximum principle}
 In this subsection, we prove the boundary behavior of the weak solutions. For this, we first establish the asymptotic behavior of the fractional $p$-Laplacian of different powers of the distance function.  
 For $\rho>0$, we define the following extension of the distance function:
 \begin{align*}
 	d_e(x)=\begin{cases}
 		{\rm dist}(x,\pa\Om) &\mbox{if }x\in\Om,\\
 		-{\rm dist}(x,\pa\Om) &\mbox{if }x\in(\Om^c)_{\rho},\\
 		-\rho &\mbox{otherwise},
 	\end{cases}
 \end{align*}
 where $(\Om^c)_{\rho}:=\{x\in\Om^c : \mathrm{dist}(x,\pa\Om)<\rho \}$. Next, for  $\al,\rho>0$ and $\ka\ge 0$, we set 
 \begin{align*}
 	\ov w_{\rho}(x)=\begin{cases}
 		(d_e(x)+\ka^{1/\al})_+^\al \quad &\mbox{if }x\in\Om\cup(\Om^c)_\rho, \\
 		0 \qquad&\mbox{otherwise}.
 	\end{cases}
 \end{align*}
 \begin{Theorem}\cite[Theorem 3.3]{arora}\label{sup}
  There exists $\ka_1,\varrho_6>0$ such that for all $\ka\in[0,\ka_1)$ and $\al\in (0,s_1)$, there exists a positive constant $C_5$ such that for all $\varrho\in(0,\varrho_6)$:
 	\begin{align*}
 		(-\De)_p^{s_1}\ov w_\rho \geq C_5 (d+\ka^{1/\al})^{-(ps_1-\al(p-1))} \quad\mbox{weakly in }\Om_\varrho.
 	\end{align*}
  Further, for all $\ka>0$ and $\al\in (0,s_1)$, $\ov w_\rho \in \widetilde{W}^{s_1,p}(\Om_{\varrho_6})$.
 \end{Theorem}

 \begin{Lemma}\label{lemA2}
  Let $\al\in [s_2,1)$ be such that $\al\neq q' s_2$. Then, there exist $\ka_2,\varrho_3>0$ such that for all $\ka\in[0,\ka_2)$ and $\varrho\in(0,\varrho_3)$: 
  $$(-\De)_q^{s_2} \ov w_\rho=h \quad\mbox{weakly in }\Om_{\varrho},$$ 
  for some $h\in L^\infty(\Om_{\varrho_3})$ (which is independent of $\ka\in (0,1)$).
 \end{Lemma}
 \begin{proof}
  Since $\Om$ has the boundary of the type $C^{1,1}$, as in \cite[Theorem 3.1]{arora}, for every $x\in\pa\Om$, there exit a neighborhood $N_x$ of $x$ and a bijective map $\Psi_x: E\to N_x$ satisfying 
 	\begin{align*}
 	  \Psi_x\in C^{1,1}(\ov E), \ \Psi_x^{-1}\in C^{1,1}(\ov N_x), \ \Psi_x(E_+)=N_x\cap \Om, \ \Psi_x(E_0)=N_x\cap \pa\Om,
 	\end{align*}
  where $E:=\{X=(X',X_N) : |X'|<1, |X_N|<1 \}$, $E_+:=E\cap \mb R^{N}_+$ and $E_0:=E\cap \{X_N=0\}$.\\
  Since $\pa\Om$ is compact, there exists a finite covering $\{B_{r_i}(x_i)\}_{i\in I}$ of $\pa\Om$ such that 
 	\begin{align*}
 	 \Om_{\varrho_3} \subset \cup_{i\in I}B_{r_i}(x_i) \quad \mbox{and } \Psi_x^{-1}(B_{r_i}(x_i))\subset B_{2\rho}(0), \mbox{ for all }i\in I,
 	\end{align*}  
  where $\varrho_3,\rho>0$ are small enough. Further, there exist diffeomorphisms $\Phi_i\in C^{1,1}(\mb R^N,\mb R^N)$ such that, for all $i\in I$, $\Phi_i = \Psi_{x_i}$ in $B_{2\rho}(0)$ and $\Phi_i= Id$ in $B_{4\rho}(0)^c$ together with 
 	\begin{align*}
 	 \Om_{\varrho_3}\cap B_{r_i}(x_i)\Subset \Phi_i(B^+_{\rho}), \quad d_e(\Phi_i(X))= (X_N+\ka^{1/\al})_+-\ka^{1/\al} \mbox{ for all }X\in B_{2\rho},
 	\end{align*}
 and for $\ka$ small enough such that $\ka^{1/\al}<\rho$, 
 \begin{align*}
 	\Phi_i \big( B_\rho \cap \{ X_N\geq - \ka^{1/\al} \} \big) \subset \Om \cup (\Om^c)_\rho,
 \end{align*}
  where $B_\rho=B_{\rho}(0)$ and $B^+_{\rho}= B_\rho\cap\mb R_{+}^N$. By the finite covering argument, it is sufficient to consider one fixed $i\in I$ and for simplicity we take $x_i=0$, $\Phi_i=\Phi$ and $\Phi(0)=0$. To prove the claim of the lemma, we will prove that 
 	\begin{align*}
 		h_\e(x) =\int_{D_\e(x)^c} \frac{[\ov w_\rho(x)-\ov w_\rho(y)]^{q-1}}{|x-y|^{N+qs_2}}dy, 
 	\end{align*}
  where $D_\e(x) = \{ y \ : \ |\Phi^{-1}(x)-\Phi^{-1}(y)|\leq \e\}$, converges to some $h$ in $L^1_{\mathrm{loc}}(\Om_{\varrho_3}\cap B_{r_i})$. We change the variable $X=\Phi^{-1}(x)$ and noting that $X\in B^+_{\rho}$ for any $x\in B_{r_i}\cap\Om_{\varrho_3}$, we get 
 	\begin{align*}
 	  h_\e(x)&=\int_{B_\e(X)^c} \frac{[\ov w_\rho(\Phi(X))-\ov w_\rho(\Phi(Y))]^{q-1}}{|\Phi(X)-\Phi(Y)|^{N+qs_2}}J_\Phi(Y)dY \\
 	  &= \left( \int_{B_\e(X)^c\cap B_{2\rho}}+ \int_{B_\e(X)^c\cap (B_{4\rho}\setminus B_{2\rho})} \right) \frac{[\ov w_\rho(\Phi(X))-\ov w_\rho(\Phi(Y))]^{q-1}}{|\Phi(X)-\Phi(Y)|^{N+qs_2}}J_\Phi(Y)dY \\
 	  & \quad + \int_{B_{4\rho}^c} \frac{[\ov w_\rho(\Phi(X))-\ov w_\rho(Y)]^{q-1}}{|\Phi(X)-\Phi(Y)|^{N+qs_2}}J_\Phi(Y)dY \\
 	  &=: \tl I_{\e}(X)+I_{\e}(X)+ I_2(X),
 	\end{align*}
 where $J_\Phi(Y)=|\mathrm{det} \na\Phi(Y)|$. We first estimate the quantity $I_2(X)$. Note that $\ov w_\rho=0$ in $(\Om\cup(\Om^c)_\rho)^c$, and elsewhere it is bounded (the bound is independent of $\ka\in (0,1)$), therefore
 	\begin{align*}
 	  I_2(X) &= \int_{B_{4\rho}^c\cap(\Om\cup(\Om^c)_\rho)} \frac{[\ov w_\rho(\Phi(X))-\ov w_\rho(Y)]^{q-1}}{|\Phi(X)-\Phi(Y)|^{N+qs_2}}J_\Phi(Y) dY \\
 	  &\quad+ \int_{B_{4\rho}^c\cap(\Om\cup(\Om^c)_\rho)^c} \frac{[\ov w_\rho(\Phi(X))^{s_1}]^{q-1}}{|\Phi(X)-\Phi(Y)|^{N+qs_2}}J_\Phi(Y)dY \\
 	  &\leq  C_{\Om,\Phi} \int_{B_{4\rho}^c} \frac{dY}{|X-Y|^{N+qs_2}}=: C(N,q,\Phi,\Phi^{-1},\rho)<\infty,
 	\end{align*}
  where in the last inequality we have used the Lipschitz continuity of $\Phi^{-1}$ together with the fact that the map $X\mapsto |\mathrm{det}\na\Phi(X)|$ is bounded in $\mb R^N$. Consequently, $I_2(X)$ is bounded. On a similar note, using the fact that $\ov w_\rho\circ\Phi$ is bounded in $B_{4\rho}$ and $J_\Phi(\cdot)$ is bounded in $\mb R^N$, we obtain that $I_\e(X)$ is bounded in $X$ (and the bound is independent of $\e$ and $\ka$). \\
  Thus, it remains to estimate only $\tl I_\e(X)$. We distinguish the following two cases.\\
 \textbf{Case (i)}: If $\al>q' s_2$.\\
  We see that the map $\ov w_\rho\circ \Phi$ is $\al$-H\"older continuous, and $J_\Phi(\cdot)$ is bounded in $\mb R^N$. Therefore, on account of the Lipschitz continuity of $\Phi^{-1}$, we deduce that
 	\begin{align*}
 	 \tl I_{\e}(X)\leq C_{\Om,\Phi} \int_{B_\e(X)^c\cap B_{2\rho}} \frac{1}{|X-Y|^{N+q(s_2-\al)+\al}}dY \leq C(N,q,\Phi,\Phi^{-1},\rho)<\infty,
 	\end{align*}
 because of $\al> q' s_2$. \\ 
 \textbf{Case (ii)}: If $\al< q' s_2$.\\
 Set $h(X,Y)= \frac{|\na\Phi(X)(X-Y)|^{N+qs_2}}{|\Phi(X)-\Phi(Y)|^{N+qs_2}}J_\Phi(Y)-J_\Phi(X)$, for $X\neq Y$. Then, using the $\al$-H\"older continuity for the map $X\mapsto (X_N+\ka^{1/\al})^\al_+$,
  we obtain
 	\begin{align}\label{eq120}
 	 \tilde{I_{\e}}(X)
 	  &= \int_{B_\e(X)^c\cap B_{2\rho}}  \frac{[(X_N+\ka^{1/\al})_+^{\al}-(Y_N+\ka^{1/\al})_+^{\al}]^{q-1}}{|\na\Phi(X)(X-Y)|^{N+qs_2}}h(X,Y) dY \nonumber\\
 	  &\quad + \int_{B_\e(X)^c\cap B_{2\rho}} \frac{[(X_N+\ka^{1/\al})_+^{\al}-(Y_N+\ka^{1/\al})_+^{\al}]^{q-1}}{|\na\Phi(X)(X-Y)|^{N+qs_2}}J_\Phi(X)dY \nonumber\\
 	  &\leq C |B_{2\rho}|^{(\al-s_2)(q-1)} \int_{B_\e(X)^c} \frac{|X-Y|^{s_2(q-1)} |h(X,Y)|}{|\na\Phi(X)(X-Y)|^{N+qs_2}}dY \nonumber\\
 	  &\quad+  \int_{B_\e(X)^c\cap B_{2\rho}} \frac{[(X_N+\ka^{1/\al})_+^{\al}-(Y_N+\ka^{1/\al})_+^{\al}]^{q-1}}{|\na\Phi(X)(X-Y)|^{N+qs_2}} J_\Phi(X) dY \nonumber \\
 	  &=:\tl I_{1,\e}(X)+\tl I_{2,\e}(X).
 	\end{align}
  We observe that $\tl I_{1,\e}(X)$ is finite due to \cite[(3.9)]{iann}. Indeed, from \cite[(3.7)]{iann}, we have $|h(X,Y)|\leq C(\Phi)\min\{|X-Y|,1\}$, consequently
  \begin{align}\label{eq123}
   \tl I_{1,\e}(X)&\leq C(\Phi,\Phi^{-1})|B_{2\rho}|^{(\al-s_2)(q-1)} \int_{B_\e(X)^c} \frac{\min\{|X-Y|,1\}}{|X-Y|^{N+s_2}}dY \nonumber\\
   &\leq C(\Phi,\Phi^{-1},\rho,N,q,s_2,\al)(\e^{1-s_2}+1)<\infty.
  \end{align} 
 For the second term in \eqref{eq120}, we have
 	\begin{align}\label{eq121}
 	 \tl I_{2,\e}(X) &=\int_{B_\e(X)^c\cap B_{2\rho}} \frac{[(X_N+\ka^{1/\al})_+^{\al}-(Y_N+\ka^{1/\al})_+^{\al}]^{q-1}}{|\na\Phi(X)(X-Y)|^{N+qs_2}} J_\Phi(X) dY \nonumber\\
 	 &= \left(\int_{B_\e(X)^c}  - \int_{B_\e(X)^c\cap B_{2\rho}^c} \right) \frac{[(X_N+\ka^{1/\al})_+^{\al}-(Y_N+\ka^{1/\al})_+^{\al}]^{q-1}}{|\na\Phi(X)(X-Y)|^{N+qs_2}} J_\Phi(X) dY \nonumber\\
 	 &=: \tl I_{3,\e}(X)+ \tl I_{4,\e}(X).
 	\end{align} 
  Using the $\al$-H\"older continuity of the map $X\mapsto (X_N+\ka^{1/\al})_+^{\al}$ and the Lipschitz nature of $\na\Phi^{-1}$ together with the boundedness of $J_\Phi(\cdot)$, we obtain 
 	\begin{align}\label{eq122}
 	  |\tl I_{4,\e}(X)| &\leq \int_{B_{2\rho}^c}  \frac{|(X_N+\ka^{1/\al})_+^{\al}-(Y_N+\ka^{1/\al})_+^{\al}|^{q-1}}{|\na\Phi(X)(X-Y)|^{N+qs_2}} |J_\Phi(X)| dY \nonumber\\
 	  &\leq C \int_{B_{\rho}(X)^c} |X-Y|^{-N-qs_2+\al(q-1)}dY =: C(\Phi,\Phi^{-1},\rho)<\infty,
 	\end{align}
 where we have used the fact that $X\in B_\rho$ and $\al<q's_2$. Therefore, to conclude that $\tl I_{\e}(X)$ is bounded in $X$, it is sufficient to prove that $\tl I_{3,\e}(X)$ is bounded in $X$ (use \eqref{eq123}, \eqref{eq121} and \eqref{eq122} in \eqref{eq120}). \\
 \textbf{Claim}:  $\tl I_{3,\e}$ is bounded uniformly in $X$.\\
 To prove the claim, we will first show that the following $1$-dimensional integral: 
  	\begin{align*}
 	 g_\e(x):=\int_{B_\e(x)^{c}} \frac{[(x+\ka^{1/\al})_+^{\al}-(y+\ka^{1/\al})_+^{\al}]^{q-1}}{|x-y|^{1+qs_2}}dy, \quad\mbox{for }x>0,
 	\end{align*}
 converges to $0$ uniformly in $K$, for all $K\Subset\mb R^+$, as $\e\to 0$. Note that for any $K\Subset\mb R^+$, there exists $\rho\in (0,1)$ such that $K\subset (\rho,\rho^{-1})$. For $x\in K$ and 
  any $\e\in\mb R$ with $\e\in(0,x)$, we have $-\ka^{1/\al}<x-\e<x+\e<\frac{(x+\ka^{1/\al})^2}{x+\ka^{1/\al}-\e}-\ka^{1/\al}$. Therefore, 
 	\begin{align}\label{eq126}
 	 g_\e(x) &= \int_{-\infty}^{-\ka^{1/\al}} \frac{(x+\ka^{1/\al})^{\al(q-1)}}{|x-y|^{1+qs_2}}dy + \int_{x+\e}^{\frac{(x+\ka^{1/\al})^2}{x+\ka^{1/\al}-\e}-\ka^{1/\al}} \frac{[(x+\ka^{1/\al})^{\al}-(y+\ka^{1/\al})_+^{\al}]^{q-1}}{|x-y|^{1+qs_2}}dy \nonumber\\
 	 &\quad+\left( \int_{-\ka^{1/\al}}^{x-\e}+\int_{\frac{(x+\ka^{1/\al})^2}{x+\ka^{1/\al}-\e}-\ka^{1/\al}}^{\infty}  \right)\frac{[(x+\ka^{1/\al})^{\al}-(y+\ka^{1/\al})_+^{\al}]^{q-1}}{|x-y|^{1+qs_2}}dy \nonumber\\
 	 &=: L_1(x)+L_2(\e,x)+L_3(\e,x).
 	\end{align}
 	By straightforward integrations, we have
 	\begin{align}\label{eq127}
 		L_1(x)= \int_{-\infty}^{-\ka^{1/\al}} \frac{(x+\ka^{1/\al})^{\al(q-1)}}{|x-y|^{1+qs_2}}dy = \frac{(x+\ka^{1/\al})^{\al(q-1)-qs_2}}{qs_2}.
 	\end{align}
  Moreover, since $\al<q's_2$ (that is $q(\al-s_2)-\al<0$), and using the $\al$-H\"older nature of the map $x\mapsto (x+\ka^{1/\al})_+^{\al}$,
 	\begin{align}\label{eq128}
 	 |L_2(x,\e)| &\leq C \int_{x+\e}^{\frac{(x+\ka^{1/\al})^2}{x+\ka^{1/\al}-\e}-\ka^{1/\al}} (y-x)^{q(\al-s_2)-\al-1} dy \nonumber\\
 	 &= C_2 \frac{(x+\ka^{1/\al})^{q(\al-s_2)-\al}}{\al-q(\al-s_2)} \frac{(x+\ka^{1/\al})^{\al-q(\al-s_2)}-(x+\ka^{1/\al}-\e)^{\al-q(\al-s_2)}}{\e^{\al-q(\al-s_2)}}.
 	\end{align} 
  Next, to estimate $L_3$, we substitute $t=\frac{y+\ka^{1/\al}}{x+\ka^{1/\al}}$, and then letting $t\mapsto1/t$ in the second integral, we deduce that
 	\begin{align*}
 	 L_3(\e,x) &=(x+\ka^{1/\al})^{q(\al-s_2)-\al}\left[ \int_{0}^{1-\frac{\e}{x+\ka^{1/\al}}}\frac{(1-t^{\al})^{q-1}}{(1-t)^{1+qs_2}}dt- \int_{1+\frac{\e}{x+\ka^{1/\al}-\e}}^{\infty}  \frac{(t^\al-1)^{q-1}}{|t-1|^{1+qs_2}}dt \right] \\
 	 &= (x+\ka^{1/\al})^{q(\al-s_2)-\al} \left[ \int_{0}^{1-\frac{\e}{x+\ka^{1/\al}}}  \frac{(1-t^{\al})^{q-1}}{(1-t)^{1+qs_2}}(1-t^{-q(\al-s_2)+\al-1})dt \right].
 	\end{align*}
  Noticing $0<s_2<\al<1$ (with $\al<q's_2$) and $t\in(0,1)$, we see that the above integrand is non-positive. Therefore,
 	\begin{align}\label{eq129}
 	 L_3(\e,x) &= (x+\ka^{1/\al})^{q(\al-s_2)-\al} \left[ \int_{0}^{1-\frac{\e}{x+\ka^{1/\al}}}  \frac{(1-t^{\al})^{q-1}}{(1-t)^{1+qs_2}}(1-t^{-q(\al-s_2)+\al-s_2+s_2-1})dt \right] \nonumber\\
 	 &\leq (x+\ka^{1/\al})^{q(\al-s_2)-\al}  \left[ \int_{0}^{1-\frac{\e}{x+\ka^{1/\al}}}  \frac{(1-t^{s_2})^{q-1}}{(1-t)^{1+qs_2}}(1-t^{s_2-1})dt  \right] \nonumber\\
 	 &=(x+\ka^{1/\al})^{q(\al-s_2)-\al}  \left[ \frac{1}{qs_2}\frac{(1-t^{s_2})^q}{(1-t)^{qs_2}}\right]_{0}^{1-\frac{\e}{x+\ka^{1/\al}}} \nonumber \\
 	 &= \frac{(x+\ka^{1/\al})^{q(\al-s_2)-\al} }{qs_2} \left[ \left(\frac{(x+\ka^{1/\al})^{s_2}-(x+\ka^{1/\al}-\e)^{s_2}}{\e^{s_2}}\right)^q-1 \right]. 
 	\end{align}
  Using \eqref{eq127}, \eqref{eq128} and \eqref{eq129} in \eqref{eq126}, we obtain
  \begin{align*}
  g_\e(x)&\leq C_2 \frac{(x+\ka^{1/\al})^{q(\al-s_2)-\al}}{\al-q(\al-s_2)} \frac{(x+\ka^{1/\al})^{\al-q(\al-s_2)}-(x+\ka^{1/\al}-\e)^{\al-q(\al-s_2)}}{\e^{\al-q(\al-s_2)}} \\
  &\quad+ \frac{(x+\ka^{1/\al})^{q(\al-s_2)-\al} }{qs_2} \left(\frac{(x+\ka^{1/\al})^{s_2}-(x+\ka^{1/\al}-\e)^{s_2}}{\e^{s_2}}\right)^q,
  \end{align*}
 where the right side terms converge to $0$ uniformly in $K$, as $\e\to 0$ (see \cite[p. 1369]{iann} for details). Furthermore, by proceeding similar to the proof of \cite[Lemma 3.4, p. 1373]{iann}, we can conclude that $\tl I_{3,\e}\to 0$,  uniformly in compact subsets of $\Phi^{-1}(\Om_{\varrho_3}\cap B_r)$, as $\e\to 0$. \\
 Therefore, collecting the information that $\tl I_\e$, $I_\e$ and $I_2$ are uniformly bounded in compact sets, we conclude that $2h_\e \to h$ in $L^1_{\rm loc}(\Om_{\varrho_3}\cap B_r)$. Hence, using the convergence result of \cite[Lemma 2.5]{iann}, we get the required result of the lemma. \QED
 \end{proof}
 \begin{Remark}
  For $\al\in (s_1-1/p,s_1)$ and $\ka=0$,  we can show that $\ov w_\rho=d^{\al}\in \widetilde{W}^{s_1,p}(\Om_{\varrho_7})$. Indeed, following the calculations of \cite[Theorem 3.1, (3.8)]{arora} (see also \cite[Proof of Lemma 3.1, p.1369]{iann}), for $a<0<b$,  we have
    \begin{align}\label{eq180}
     \int_{a}^{b}\int_{a}^{b} \frac{|U(x)-U(y)|^p}{|x-y|^{1+s_1p}}dxdy<\infty,
    \end{align}
  where $U(x)=(x_+)^\al$, for $\al>s_1-1/p$. We can get a finite covering $\{B_{R_i}(x_i)\}_{i\in I}$ of $\pa\Om$ and construct a similar diffeomorphisms $\Phi_i$, as in \cite[Theorem 3.3]{arora}, such that
    \begin{align*}
  	  &\Om_{\varrho_7}\cap K_i \subset \Om_{\varrho_7}\cap B_{R_i}(x_i) \Subset\Phi_i(B_\rho^+), \\
  	  &d(\Phi_i(X))= (X_N)_+ \quad\mbox{for all }X\in\Phi_i^{-1}(K_i)\subset B_\rho,
  	 \end{align*}
  where $K_i:=B_{\tau_i}(x_i)\subset B_{R_i}(x_i)$ is constructed in such a manner that it satisfies \cite[(3.33), (3.34)]{arora} and $\rho>0$ with $B_{R_i}(x_i)\Subset\Phi_i(B_\rho^+)$. Thus, the estimate of the Sobolev norm $[\cdot]_{W^{s_1,p}(\Om_{\varrho_7})}$ of $d^\al$ can be split into (3.36) and (3.37) of \cite{arora}, and then using \eqref{eq180}, the claim follows.
  
 \end{Remark}

\begin{Proposition}\label{upper}
 Let $u\in W^{s_1,p}_0(\Om)\cap L^\infty_{\rm loc}(\Om)$ be such that $|(-\De)_p^{s_1}u+  (-\De)_q^{s_2} u| \leq K$, weakly in $\Om$, for some $K>0$. Then, for all $\sg\in (0,s_1)$, there exists a constant $\Ga>0$ (depending on data of the problem and $\|u\|_{L^\infty_{\rm loc}(\Om)}$ only) such that
 	\[ |u| \leq \Ga d^{\sg} \quad\mbox{in } \Om. \]
 \end{Proposition}
 \begin{proof}
 From Theorem \ref{sup}, with $\ka=0$, $\al=\sg\in(s_1-\frac{1}{p},s_1)$ (for deduction to the smaller $\sg$, see end of the proof), there exists $\varrho_1>0$ such that 
 	\begin{align*}
 	(-\De)_p^{s_1} d^{\sg} \geq C_5 d^{-(ps_1-\sg(p-1))} \quad\mbox{weakly in }\Om_{\varrho_1}.
 	\end{align*} 
 If $\sg\neq q' s_2$, then on account of lemma \ref{lemA2}, for the choice $\ka=0$ there, we have 
 	\begin{align*}
 	(-\De)_q^{s_2} d^{\sg} =h \quad\mbox{weakly in }\Om_{\varrho_3},
 	\end{align*}
  for some $h\in L^\infty(\Om_{\varrho_3})$. Therefore, for $\Ga>1$ (to be specified later), we obtain 
 	\begin{align*}
 	(-\De)_p^{s_1} (\Ga d^{\sg})+ (-\De)_q^{s_2} (\Ga d^{\sg}) \geq C_5 \Ga^{p-1} d^{-(ps_1-\sg(p-1))}+ \Ga^{q-1} h \quad\mbox{weakly in }\Om_{\varrho_3}.
 	\end{align*}
 	That is, 
 	\begin{align*}
 	(-\De)_p^{s_1} (\Ga d^{\sg})+  (-\De)_q^{s_2} (\Ga d^{\sg}) &\geq \Ga^{q-1} \big( C_5 d^{-(ps_1-\sg(p-1))} -  \|h \|_{L^\infty(\Om_{\varrho_3})} \big) \\
 	&\geq \Ga^{q-1} \frac{C_5}{2} d^{-(ps_1-s_2(p-1))},
 	\end{align*}
  weakly in $\Om_{\varrho_2}$, where $0<\varrho_2 < \min\{ \varrho_1, \varrho_3, \big( \frac{C_5}{2  \|h\|_{L^\infty(\Om_{\varrho_3})}} \big)^{1/(ps_1-\sg(p-1))}\}$. Now we can choose $\Ga>1$ large enough so that 
 	\begin{align}\label{eqA4}
 	(-\De)_p^{s_1} (\Ga d^{\sg})+  (-\De)_q^{s_2} (\Ga d^{\sg}) \geq \Ga^{q-1} \frac{C_5}{2} d^{-(ps_1-\sg(p-1))} \geq K \quad\mbox{weakly in }\Om_{\varrho_2}.
 	\end{align}
  Furthermore, for $\Ga$ large enough (note that $\Om\setminus\Om_{\varrho_2}\Subset\Om$), 
 	\begin{align}\label{eqA5}
 	\| u \|_{L^\infty(\Om\setminus\Om_{\varrho_2})}  \leq \Ga \varrho_2^{\sg} \leq \Ga d^{\sg} \quad\mbox{in }\Om\setminus\Om_{\varrho_2}.
 	\end{align}
  Therefore, on account of \eqref{eqA4}, \eqref{eqA5} and the weak comparison principle in $\Om_{\varrho_2}$ (see e.g. \cite[Proposition 2.6]{DDS} by noticing that $u,d^\sg\in \widetilde{W}^{s_1,p}(\Om_{\varrho_2})$), we get $u \leq \Ga d^{\sg}$ in $\Om$. If $\sg\in (0,s_1)$ is such that $\sg=q's_2$, we can choose $\sg_1\in(\sg,s_1)$ (note that $s_1>q's_2$). Then, repeating the process for $\sg_1$, we obtain
  \begin{align*}
  	u \leq \Ga_1 d^{\sg_1} \leq \Ga_1 (\mathrm{diam}(\Om))^{\sg_1-\sg} d^{\sg}:= \Ga d^{\sg} \quad\mbox{in }\Om.
  \end{align*}
 Proceeding similarly for $-u$, we get the required bound of the proposition.\QED
 \end{proof}

\textbf{Proof of Theorem \ref{bdryreg}}: 	The proof of the theorem is standard and follows on the similar lines of the proof of \cite[Theorem 1.1]{iann} by taking into account theorem \ref{impintreg} and proposition \ref{upper}. Moreover, \eqref{holderbd} follows from the interior bound on $u$, given by theorem \ref{impintreg}, and the boundary behavior given by proposition \ref{upper}.\QED

\textbf{Proof of Corollary \ref{cornonpb}}: On account of theorem \ref{mainbound}, we observe that  
 $$|f(x,u)|\leq C_0 (1+|u|^{p^*_{s_1}-1}) \leq C_0 (1+\|u\|_{L^\infty(\Om)}^{p^*_{s_1}-1})=:K>0.$$ 
Consequently, the proof of the corollary follows from theorem \ref{bdryreg}.\QED

\textbf{Proof of Theorem \ref{strongmax}}:
  As a consequence of the weak comparison principle, we have $u\geq 0$ in $\mb R^N$. Since $u\in C(\ov\Om)$, we see that $U_0:=\{x\in\Om \ : \ u(x)=0\}$ is a closed set. By the weak Harnack inequality for non-negative super-solutions (see \cite[Lemma 2.7]{DDS}, which is valid also for $K=0$ there), we have 
 	\begin{align*}
 	\inf_{B_{r/4}} u \geq \sg \left(\Xint-_{B_r\setminus B_{r/2}} u^{p-1} \right)^{1/(p-1)} \quad\mbox{for all }B_r\subset\Om,
 	\end{align*} 
  where $\sg\in (0,1)$. This shows that the set $U_0$ is open in $\Om$. Suppose $u\not\equiv 0$, otherwise there is nothing to prove. Let $\Om=\cup_{i\in\mc I} \Om_i$, where $\Om_i$'s are the connected components of $\Om$. Since $U_0$ is closed as well as open in $\Om$, for each $i\in\mc I$, either $\Om_i\subset U_0$ or $\Om_i \cap U_0 =\emptyset$, that is, either $u$ vanishes in $\Om_i$ or $u>0$ in $\Om_i$. Since $u\not\equiv 0$ in $\Om$, the exits $\Om_1$ such that $u>0$ in $\Om_1$. If $u>0$ in all of $\Om_i$, then we are done. Otherwise, there exists a connected component $\Om_2$ such that $u=0$ in $\Om_2$. Let $\psi\in C_c^\infty(\Om_2)$ such that $\psi\geq 0$ and $\psi\not\equiv 0$. Since $u$ is a weak super-solution, we obtain 
 	\begin{align*}
 	0 &\leq \int_{\mb R^N} \frac{[u(x)-u(y)]^{p-1}(\psi(x)-\psi(y))}{|x-y|^{N+ps_1}}dxdy +  \int_{\mb R^N} \frac{[u(x)-u(y)]^{q-1}(\psi(x)-\psi(y))}{|x-y|^{N+qs_2}}dxdy \\
 	&=-2 \int_{\Om_2}\psi(x) \int_{\Om_2^c} \frac{u(y)^{p-1}}{|x-y|^{N+ps_1}}-2 \int_{\Om_2}\psi(x) \int_{\Om_2^c} \frac{u(y)^{q-1}}{|x-y|^{N+qs_2}}\\
 	&<0,
 	\end{align*} 
  which is a contradiction. This completes the proof.\QED

 Next, we prove a Hopf type maximum principle for non-homogeneous fractional $(p,q)$-Laplacian type operators.

 \textbf{Proof of Proposition \ref{hopf}}:
  By the strong maximum principle of Theorem \ref{strongmax}, we infer that either $u=0$ in $\mb R^N$ or $u>0$ in $\Om$. Suppose $u\not\equiv 0$ in $\Om$. Using Lemma \ref{lemA2} (for $\al=s_1$ and $\ka=0$) and \cite[Theorem 3.6]{iann} (for $s=s_1$ there), there exist $\varrho_4>0$ and $\tilde{g},g\in L^\infty(\Om_{\varrho_4})$ such that 
 	\begin{align*}
 		(-\De)_p^{s_1} d^{s_1}=\tilde{g} \quad \mbox{and }  (-\De)_q^{s_2} d^{s_1}=g \quad\mbox{weakly in }\Om_{\varrho_4}.
 	\end{align*}
  Let $B\subset \Om\cap\Om_{\varrho_4}^c$ be a closed set and $\eta>0$ be a constant (to be specified later). On account of \cite[Lemma 2.5]{DDS}, for $w= d^{s_1}+\eta \chi_B$, we have 
 	\begin{align}\label{eqA3}
 	 (-\De)_p^{s_1} w = (-\De)_p^{s_1} d^{s_1} + h_{\eta,p} \quad\mbox{and } (-\De)_q^{s_2} w= (-\De)_q^{s_2} d^{s_1}+h_{\eta,q}, \quad\mbox{weakly in }\Om_{\varrho_4},
 	\end{align}
  where $h_{\eta,p}(x):= 2\int_{B} \frac{[d^{s_1}(x)-d^{s_1}(y)-\eta]^{p-1}-[d^{s_1}(x)-d^{s_1}(y)]^{p-1}}{|x-y|^{N+ps_1}}dy$, for a.e. $x\in\Om_{\varrho_4}$, and  $h_{\eta,q}$ is defined analogously. Noticing the fact that $d^{s_1}\in L^\infty(\Om)$ and $\mathrm{dist}(B,\Om_{\varrho_4})>0$, we have 
 	\begin{align*}
 	 h_{\eta,p}(x), \ h_{\eta,q}(x) \to -\infty \ \mbox{ uniformly in }\Om_{\varrho_4}, \ \mbox{ as }\eta\to\infty.
 	\end{align*}
 Now, we choose $\eta$ large enough such that 
 	\begin{align}\label{eqA2}
 		\sup_{\Om_{\varrho_4}} (\tilde{g}+h_{\eta,p}) \leq 0 \quad\mbox{and } \sup_{\Om_{\varrho_4}} (g+h_{\eta,q}) \leq 0.
 	\end{align}
  Since $u\in C(\ov\Om)$ and $u>0$, we can choose $c\in (0,1)$ such that 
 	\[ c (\mathrm{diam}(\Om)^{s_1}+\eta) < \inf_{\Om\setminus\Om_{\varrho_4}} u(x).\]
  Thus, $v =c w \leq  u$ in $\mb R^N\setminus\Om_{\varrho_4}$. Using \eqref{eqA3}, \eqref{eqA2} together with the fact that $u$ is a weak super-solution, we deduce that
 	\begin{align*}
 	 (-\De)_p^{s_1} v +  (-\De)_q^{s_2} v &= c^{p-1} (\tilde{g}+ h_{\eta,p}) +  c^{q-1} (g+ h_{\eta,q}) \\
 	 &\leq c^{p-1} \sup_{\Om_{\varrho_4}} (\tilde{g}+ h_{\eta,p}) +  c^{q-1} \sup_{\Om_{\varrho_4}} (g+ h_{\eta,q}) \\
 	 &\leq 0 \leq (-\De)_p^{s_1} u +  (-\De)_q^{s_2} u, 
 	\end{align*}
  weakly in $\Om_{\varrho_4}$. Therefore, by the weak comparison principle in $\Om_{\varrho_4}$, we obtain $c w \leq u$ in $\Om_{\varrho_4}$. Consequently, 
 	\begin{align*}
 	 c \leq \frac{u(x)}{d^{s_1}(x)} \quad\mbox{ for all }x\in \Om_{\varrho_4}.
 	\end{align*}
 This completes the proof of the proposition. \QED

\textbf{Proof of Theorem \ref{strongcompreg}}:
 The proof of the theorem is essentially the same as the one of  \cite[Theorem 2.7]{iannmospap}. Indeed, by continuity and the fact that $u\not\equiv v$, we can find $x_0\in\Om$, $\rho,\e>0$ and  such that $B_\rho(x_0)\subset\Om$ and 
	\begin{align}\label{eq190}
		\sup_{B_\rho(x_0)}  v < \inf_{B_\rho (x_0)} u -\e/2.
	\end{align}
  We take $\Ga>1$ and define 
	\begin{align*}
		w_\Ga(x)=\begin{cases}
			\Ga v(x) \quad\mbox{if }x\in B_{\rho/2}^c(x_0)\\
			u(x) \quad\mbox{if }x\in B_{\rho/2}(x_0),
		\end{cases}
	 \mbox{for all }x\in\mb R^N.
	\end{align*}
 Moreover, on account of \cite[Lemma 2.5]{DDS}, we have, weakly in $\Om\setminus B_\rho(x_0)$,
 \begin{align*}
 	(-\De)_p^{s_1}w_\Ga +(-\De)_q^{s_2} w_\Ga \leq \Ga^{p-1} (-\De)_p^{s_1}v + \Ga^{q-1}(-\De)_q^{s_2}v-C_1 \e^{p-1}-C_2 \e^{q-1}.
 \end{align*}
 That is,
 {\small\begin{align*}
 	(-\De)_p^{s_1}w_\Ga +(-\De)_q^{s_2} w_\Ga &\leq \Ga^{p-1} \big((-\De)_p^{s_1}v +(-\De)_q^{s_2}v \big)+ (\Ga^{q-1}- \Ga^{p-1}) (-\De)_q^{s_2}v-C_1 \e^{p-1}-C_2 \e^{q-1} \\
 	&\leq (-\De)_p^{s_1}u +(-\De)_q^{s_2} u + (\Ga^{p-1}-1)K+ (\Ga^{p-1}- \Ga^{q-1})K_1-C_1 \e^{p-1}-C_2 \e^{q-1}.
 \end{align*}}
Since $\Ga^{p-1},\Ga^{q-1}\to 1$, as $\Ga\to 1$, we can choose $\Ga>1$ (close to $1$) such that 
 \begin{align*}
 	(-\De)_p^{s_1}w_\Ga +(-\De)_q^{s_2} w_\Ga 
 	\leq (-\De)_p^{s_1}u +(-\De)_q^{s_2} u, 
 \end{align*}
 weakly in $\Om\setminus B_\rho(x_0)$. Therefore, by the weak comparison principle, we have $w_\Ga\leq u$ in $\Om$. Hence, on account of \eqref{eq190}, we obtain $u\geq \Ga v>v$ in $\Om$. Moreover, by using proposition \ref{hopf},  $\frac{u-v}{d^{s_1}}\geq \frac{(\Ga-1)v}{d^{s_1}}\geq C>0$ in $\Om$. \QED

 As an application to the above results, we have the following Theorem. The result will be used to construct a sub-solution for singular problems, particularly for the equations involving non-homogeneous type operators. 
\begin{Theorem}\label{sub}
 Suppose $2\leq q\leq p<\infty$ and $0<s_2\leq s_1<1$. For every $\vartheta>0$, there exists a unique solution $w_{\vartheta}\in W^{s_1,p}_0(\Om)\cap C^{0,\sg}(\ov\Om)$, for all $\sg\in (0,s_1)$, of the following problem:
	\begin{equation*}
	 \left\{\begin{array}{rllll}	  
	 (-\Delta)^{s_1}_{p}u+  (-\Delta)^{s_2}_{q}u &= \vartheta, \; \; u>0 \quad \mbox{in } \Om, \\
	     u &=0 \quad \mbox{in }  \mb R^N\setminus \Om.
	\end{array}
	\right.\tag{$E_\vartheta$}\label{probsub}
	\end{equation*}
 Moreover, $w_\vartheta\to 0$ in $C^{0,\sg}(\ov\Om)$, as $\vartheta\to 0$, for all $\sg<s_1$.
\end{Theorem}
\begin{proof}
 The existence of a solution in $W^{s_1,p}_0(\Om)$ follows by using the standard minimization argument. Indeed, the associated energy functional is weakly lower semicontinuous and coercive in $W^{s_1,p}_0(\Om)$. Hence there exists a global minimizer, which in fact will be a solution (due to the nice structure of the functional). Uniqueness follows from the monotonocity of the operator $(-\De)_p^{s_1}+ (-\De)_q^{s_2}$ (see for instance \cite[Proof of Lemma 4.2]{DDS}). Moreover, boundedness of the solution is a consequence of \cite[Theorem 2.3]{DDS}, and then using Theorem \ref{bdryreg}, we see that $w_{\vartheta}\in C^{0,\sg}(\ov\Om)$, for all $\sg\in (0,s_1)$. Further, we see that $\|w_{\vartheta}\|_{L^\infty(\Om)}, \|w_\vartheta\|_{W^{s_1,p}_0(\Om)}$ are bounded and independent of $\vartheta\in (0,1)$. Therefore, again employing Theorem \ref{bdryreg}, we get that 
    \[ \| w_\vartheta \|_{C^\sg(\ov\Om)} \leq C,  \]
 where $C>0$ is a constant, independent of $\vartheta$. By the compact embedding result (Arzela-Ascoli's theorem), we have that $w_\vartheta\to w$ in $C^{0,\sg_1}(\ov\Om)$, for all $\sg_1<\sg$. On account of the uniqueness result for the solution to \eqref{probsub}, for $\vartheta=0$, we conclude that $w=0$. Additionally, by the weak comparison principle, we have $w_{\vartheta_1} \leq w_{\vartheta_2}$ in $\Om$, for $\vartheta_1 < \vartheta_2$. \QED
\end{proof}

\section{Regularity results for singular problems}\label{singl}
 In this section, we obtain some regularity results for the weak solution to problem \eqref{probsing}. For the existence of the minimal solution,
 we consider the following auxiliary problem, for $\e>0$:
 \begin{equation*}
    \left\{\begin{array}{rllll}
     (-\Delta)^{s_1}_{p}u+(-\Delta)^{s_2}_{q}u & = K_{\ga,\e}(x)(u+\e)^{-\de}, \; \ u>0 \quad \text{in } \Om, \\ 
     u&=0 \quad \text{in } \mathbb{R}^N\setminus \Om,
    \end{array}
    \right.\tag{$\mc S^\e_{\ga,\de}$} \label{probsingeps}
 \end{equation*}
where $K_{\ga,\e}(x):=\begin{cases}
  (K_{\ga}(x)^\frac{-1}{\ga}+\e^\frac{1}{\al_{\ga,\de}})^{-\ga}  & \mbox{if } K_\ga(x)>0,\\
  0 &\mbox{ otherwise},
 \end{cases}$
 with $\al_{\ga,\de}=\frac{ps_1-\ga}{p-1+\de}$. Then, we have 
 \begin{align}\label{keraux}
 	 C_3 (d(x)+\e^\frac{1}{\al_{\ga,\de}})^{-\ga} \leq K_{\ga,\e}(x)\leq  C_4 (d(x)+\e^\frac{1}{\al_{\ga,\de}})^{-\ga} \quad\mbox{in }\Om.
 \end{align}
 We have the existence and uniqueness result for the auxiliary problem as below.
 \begin{Proposition}\label{prop1}
 Let $1<q\leq p<\infty$. Then, there exists a unique solution $v_\e\in W^{s_1,p}_0(\Om)$ of problem \eqref{probsingeps}, for all $\e>0$, $\ga\ge 0$ and $\de>0$. Moreover, the sequence $\{v_\e\}$ is decreasing in $\e$ and for every $\Om^\prime\Subset\Om$, there exists a constant $C_{\Om^\prime}>0$ such that 
 \begin{align}\label{eq20}
 	C_{\Om^\prime}\leq v_1(x)\leq v_\e(x) \quad\mbox{in }\Om^\prime.
  \end{align}	
 Further, the sequence $\{v_\e\}$ is bounded in $W^{s_1,p}_0(\Om)$ if $\ga-s_1(1-\de)\leq 0$, and the sequence $\{v_\e^\theta\}$ is bounded in $W^{s_1,p}_0(\Om)$, if $\ga-s_1(1-\de)> 0$, for some $\theta>0$.
 \end{Proposition}
\begin{proof}
 Proceeding similar to \cite{JDS}, we can prove the existence and uniqueness result. Moreover, by the regularity result of theorem \ref{bdryreg}, the strong maximum principle \ref{strongmax} and the weak comparison principle, we obtain \eqref{eq20}. Next, we prove the boundedness of the sequence $\{v_\e\}$. For the case $\ga-s_1(1-\de)\leq 0$, testing the weak formulation of the problem by $v_\e$, and using \eqref{keraux} together with the fractional Hardy's inequality,  we get 
 \begin{align*}
 \|v_\e\|^p_{W^{s_1,p}_0(\Om)}\leq A_p(v_\e,v_\e,\mb R^{2N})+ A_q(v_\e,v_\e,\mb R^{2N})&=\int_{\Om} K_{\ga,\e}(x)(v_\e+\e)^{-\de}v_\e\\
 &\leq  C_4 \int_{\Om}d^{-\ga+s_1(1-\de)}\left(\frac{v_\e}{d^{s_1}}\right)^{1-\de}\\
 &\leq C\|v_\e\|^{1-\de}_{W^{s_1,p}_0(\Om)},
 \end{align*}
 where $C$ is independent of $\e$. For the remaining case, we test the weak formulation of the problem by $v_\e^{p(\theta-1)+1}$, for some $\theta\geq 1$ (to be specified later),
 \begin{align*}
 	\|v_\e^\theta\|^p_{W^{s_1,p}_0(\Om)}\leq C A_p(v_\e,v_\e^{p(\theta-1)+1},\mb R^{2N})+ C A_q(v_\e,v_\e^{p(\theta-1)+1},\mb R^{2N})
 	&=C\int_{\Om} \frac{K_{\ga,\e}(x)}{(v_\e+\e)^{\de}}v_\e^{p(\theta-1)+1},
 \end{align*}
 where we have used \cite[Lemma A.2]{brasco2} for $g(t):=t^{p(\theta-1)+1}$, $G(t):=\int_{0}^{t}g^\prime(\tau)^{1/p}d\tau$ and $\tl G(t):=\int_{0}^{t}g^\prime(\tau)^{1/q}d\tau$, which implies that $A_q(v_\e,v_\e^{p(\theta-1)+1},\mb R^{2N})\geq \| \tl G(v_\e)\|_{W^{s_2,q}_0(\Om)} \geq 0$. Therefore, on account of \eqref{keraux}, H\"older's and Hardy's inequalities, we obtain
 \begin{align*}
  \|v_\e^\theta\|^p_{W^{s_1,p}_0(\Om)} &\leq C C_4 \int_{\Om}d^{-\ga+s_1\frac{p(\theta-1)+1-\de}{\theta}} \left(\frac{v_\e^\theta}{d^{s_1}}\right)^\frac{p(\theta-1)+1-\de}{\theta}\\
  &\leq C \|v_\e^\theta\|^{\frac{p(\theta-1)+1-\de}{\theta}}_{W^{s_1,p}_0(\Om)},
 \end{align*}
 where $\theta>\max\{1, \frac{(p+\de-1)(s_1-1/p)}{ps_1-\ga}, \frac{p+\de-1}{p}\}$ and $C>0$ is independent of $\e$. This concludes that the sequence $\{v_\e^\theta\}$ is bounded in $W^{s_1,p}_0(\Om)$. \QED
\end{proof}
\begin{Corollary}\label{cor3}
  For all $\ga\in(0,ps_1)$, up to a subsequence, $v_\e$ converges pointwise to $v$, where $v$ is the minimal solution to problem \eqref{probsing}. 
\end{Corollary}
Next, we state the following weak comparison principle for singular problems whose proof runs essentially along the same lines of \cite[Theorem 4.2]{canino} and \cite[Theorem 1.5]{JDS}.
\begin{Proposition}\label{prop4}
	Let $\ga<1+s_1-\frac{1}{p}$ and $u,v\in W^{s_1,p}_{loc}(\Om)$ be sub and super solution of \eqref{probsing}, respectively, in the sense of definition \ref{defn1}. Then, $u\le v$ a.e. in $\Om$.
\end{Proposition}

\begin{Theorem}\label{uppdist}
 Let $1<q\leq p<\infty$, $\ga\in[0,ps_1)$ and $\de>0$. Let $v\in W^{s_1,p}_{\rm loc}(\Om)$ be the minimal solution to problem \eqref{probsing} (given by corollary \ref{cor3}). Then, there exist positive constants $\eta,\Ga>0$ (depending on the data of the problem and $\|v\|_{L^\infty_{\rm loc}(\Om)}$ only) such that 
	\begin{align*}
	 \eta d(x)^{s_1} \leq v(x) \leq \Ga d(x)^\mu \quad \mbox{in }\Om, 
	\end{align*}
 where $\mu= \begin{cases}
  \frac{ps_1-\ga}{p-1+\de}  \quad\mbox{if }	\ga-s_1(1-\de)> 0 \mbox{ with }\ga\neq ps_1-q's_2(p-1+\de),\\	
  \frac{ps_1-\ga_1}{p-1+\de}  \quad \mbox{for all }\ga_1\in (\ga,ps_1), \mbox{ if }	\ga-s_1(1-\de)> 0 \mbox{ with }\ga\neq ps_1-q's_2(p-1+\de),\\
  \sg \quad\mbox{for all }\sg\in(0,s_1), \mbox{ if }\ga-s_1(1-\de)\leq 0,	
 \end{cases}$ and the lower estimate holds only for $s_1\neq q' s_2$.
\end{Theorem}
\begin{proof}
 For $\vartheta>0$ sufficiently small, we see that the solution $w_\vartheta\in W^{s_1,p}_0(\Om)$ to problem \eqref{probsub} (given by Theorem \ref{sub}) is a subsolution to $(\mc S^{1}_{\ga,\de})$. Therefore, by the weak comparison principle (note that $v_1\in W^{s_1,p}_0(\Om)$), we get $w_\vartheta \leq v_1$ in $\Om$. Consequently, by Hopf's maximum principle (see proposition \ref{hopf}) and the fact that $\{v_\e\}_\e$ is decreasing in $\e$ (proposition \ref{prop1}), we obtain
 $$\eta d^{s_1} \leq v_1 \leq v_\e \quad\mbox{in }\Om,$$ 
 for some constant $\eta>0$ independent of $\e$. Passing to the limit $\e\to 0$, we get the required lower bound on $v$. \\
 To prove the upper bound we distinguish the following cases.\\
 \textit{Case (i)}: If $\ga-s_1(1-\de)>0$ with $\ga\neq ps_1-q's_2(p-1+\de)$.\\
 Fix $\Ga>1$ (to be specified later), then using Theorem \ref{sup} and Lemma \ref{lemA2}, for $\al=\frac{s_1p-\ga}{p-1+\de}\in(0,s_1)$ and $\ka=\e$ there, we have 
 \begin{align*}
 	(-\De)_p^{s_1}(\Ga\ov w_\rho) \geq C_5 \Ga^{p-1}(d+\e^{1/\al})^{-(ps_1-\al(p-1))} \quad\mbox{and } 
 	(-\De)_q^{s_2} (\Ga\ov w_\rho) = \Ga^{q-1} h, 
 \end{align*}
  weakly in $\Om_\varrho$, for some $\varrho>0$ sufficiently small and $h\in L^\infty(\Om_{\varrho})$ (independent of $\e$). Therefore, choosing $\varrho,\e_*>0$ sufficiently small (depending on $h$ also), and  using \eqref{keraux}, we deduce that
  \begin{align}\label{eq25}
  (-\De)_p^{s_1}(\Ga\ov w_\rho)+  (-\De)_q^{s_2} (\Ga\ov w_\rho) &\geq C_5 \Ga^{p-1}(d+\e^{1/\al})^{-\ga-\al\de}- \Ga^{q-1} 
  \| h \|_{L^\infty(\Om_{\varrho})} \nonumber\\
  &\geq  \Ga^{p-1} \frac{C_5}{2} (d+\e^{1/\al})^{-\ga-\al\de} \nonumber\\
  &\geq \Ga^{p-1} C_6  C_4^{-1} K_{\ga,\e}(x) (d+\e^{1/\al})^{-\al\de},
  \end{align}
  weakly in $\Om_\varrho$, for all $\e\in (0,\e_*)$. Furthermore, in $\Om_{\varrho}$, we have
  \begin{align*}
  	\Ga^{p-1} C_6  C_4^{-1} (d+\e^{1/\al})^{-\al\de} \geq (\Ga (d+\e^{1/\al})^{\al})^{-\de} \geq (\Ga (d+\e^{1/\al})^{\al}+\e)^{-\de} = (\Ga \ov w_\rho+\e)^{-\de},
  \end{align*}
  provided $\Ga^{p-1+\de} C_6  C_4^{-1}\ge 1$. Thus, from \eqref{eq25}, we obtain 
    \begin{align}\label{eq26}
     (-\De)_p^{s_1}(\Ga\ov w_\rho)+ (-\De)_q^{s_2} (\Ga\ov w_\rho) &\geq K_{\ga,\e}(x)(\Ga \ov w_\rho+\e)^{-\de} \quad\mbox{weakly in }\Om_\varrho.
    \end{align}
  On account of \eqref{keraux} and \eqref{eq20}, we note that 
    \begin{align}\label{eqS1}
   	 K_{\ga,\e}(x)(v_\e+\e)^{-\de} \leq  C_4 d(x)^{-\ga} v_\e^{-\de} \leq  C_4 \varrho^{-\ga} C_{\varrho}^{-\de} \quad\mbox{in } \Om\setminus\Om_\varrho,
    \end{align}
  that is, the right hand side quantity in \eqref{probsingeps} is independent of $\e$ in $\Om\setminus\Om_\varrho$. Next, we claim that $v_\e$ is bounded, independent of $\e$, in $\Om\setminus\Om_{\varrho}$. Let $\{ B_{\varrho/4}(x_i) \}_{i=1,\dots,m}$ be a finite covering of $\ov{\Om\setminus\Om_{\varrho}}$ such that
    \[  \ov{\Om\setminus\Om_{\varrho}} \subset \cup_{i=1}^{m} B_{\varrho/4}(x_i)\subset \Om\setminus\Om_{\varrho/2}\Subset\Om.\] 
  Therefore, from Propositions \ref{localbdd} and Corollary \ref{corlocalbdd}, we deduce that 
  \begin{align}\label{eqS2}
  	\| v_\e \|_{L^\infty(B_{\varrho/4}(x_i))} \leq C \Big( \Xint-_{B_{\varrho/2}(x_i)}v_\e^p dx \Big)^{1/p}+ \sum_{(l,s)} T_{p-1}(v_\e;x_i,\frac{\varrho}{4})^{(l-1)/(p-1)}+ C, 
  \end{align}
 where $C=\big(C(N,p,q,s_1)(1+\|K_{\ga,\e}(x)(v_\e+\e)^{-\de}\|_{L^\infty(B_{\varrho/2})})\big)^{p^*_{s_1}/p^2}$, with $C(N,p,q,s_1)>0$ as a constant. On account of \eqref{eqS1}, we see that $C>0$ is independent of $\e$. For the other terms, we have from proposition \ref{prop1} that $\{v_\e^\theta\}_\e$ is bounded in $W^{s_1,p}_0(\Om)$, for some $\theta\geq 1$. Therefore,
  \begin{align*}
  	\Xint-_{B_{\varrho/2}(x_i)}v_\e^p dx \leq C (1+\| v_\e^\theta \|_{L^p(\Om)}) \leq C (1+ \| v_\e^\theta \|_{W^{s_1,p}_0(\Om)})\leq C,
  \end{align*}
 and noting that $v_\e=0$ in $\mb R^N\setminus\Om$, for $(l,s)\in\{(p,s_1),(q,s_2)\}$,
  \begin{align*}
  	T_{l-1}(v_\e;x_i,\varrho/4)^{l-1}\leq C \varrho^{ls}\int_{\Om\setminus B_{\varrho/4}(x_i)}\frac{|v_\e(x)|^{l-1}}{\varrho^{N+sl}}dx\leq C \varrho^{-N} \| v_\e \|^{p-1}_{L^{p-1}(\Om)}\leq C.
  \end{align*}
 Hence, \eqref{eqS2} implies that the sequence
 $\{ \|v_\e\|_{L^\infty(\Om\setminus\Om_{\varrho})} \}_\e$ is uniformly bounded with respect to $\e$, that is, $\|v_\e\|_{L^\infty(\Om\setminus\Om_\varrho)}  \leq C'_{\varrho}$. Consequently, we can choose $\Ga>1$ large enough and independent of $\e$ such that 
   \[ v_\e \leq \|v_\e\|_{L^\infty(\Om\setminus\Om_\varrho)} \le C'_{\varrho} \leq \Ga (d+\e^{1/\al})^\al= \Ga \ov w_\rho \quad\mbox{in } \Om\setminus\Om_\varrho, \]
  this together with \eqref{eq26} and the weak comparison principle in $\Om_{\varrho}$ (note that $v_\e\in W^{s_1,p}_0(\Om)\cap L^\infty(\mb R^N)$ implies that $v_\e\in\widetilde{W}^{s_1,p}(\Om_{\varrho})\cap\widetilde{W}^{s_2,q}(\Om_{\varrho})$) yields 
  \[ v_\e \leq \Ga \ov w_\rho \quad\mbox{in }\Om, \ \ \mbox{for all } \e\in(0,1). \]
  Passing to the limit $\e\to 0$ in the above expression, we get the required upper bound for $v$, in this case. \\
\textit{Case (ii)}: If $\ga-s_1(1-\de)> 0$ with $\ga= ps_1-q's_2(p-1+\de)$.\\ 
 Similar to case (i) above,  we fix $\Ga>1$ (to be specified later) and choose $\ga_1\in (\ga,ps_1)$. Then, $\ga_1-s_1(1-\de)> 0$ and $\ga_1\neq ps_1-q's_2(p-1+\de)$. 
  Thus using Theorem \ref{sup} and Lemma \ref{lemA2}, for $\al_1=\frac{s_1p-\ga_1}{p-1+\de}\in(0,s_1)$ and $\ka=\e$ there, we have 
    \begin{align}\label{eq25'}
  	(-\De)_p^{s_1}(\Ga\ov w_\rho)+  (-\De)_q^{s_2} (\Ga\ov w_\rho)
  	&\geq  \Ga^{p-1} \frac{C_5}{2} (d+\e^{1/\al_1})^{-\ga-\al_1\de} \nonumber\\
  	&\geq \Ga^{p-1} C_6  C_4^{-1} K_{\ga,\e}(x) (d+\e^{1/\al_1})^{-\ga_1+\ga} (d+\e^{1/\al_1})^{-\al_1\de},
  \end{align}
  weakly in $\Om_\varrho$, for all $\e\in (0,\e_*)$. 
   Furthermore, in $\Om_{\varrho}$, we have
  \begin{align*}
  	\Ga^{p-1} C_6  C_4^{-1} (d+\e^{1/\al_1})^{-\ga_1+\ga} (d+\e^{1/\al})^{-\al_1\de} 
  	&\geq \Ga^{p-1} C_6  C_4^{-1} (\varrho+1)^{-\ga_1+\ga} (d+\e^{1/\al_1})^{-\al_1\de} \\
  	&\geq (\Ga (d+\e^{1/\al_1})^{\al_1})^{-\de} 
  	\geq (\Ga \ov w_\rho+\e)^{-\de},
  \end{align*}
  provided $\Ga^{p-1+\de} C_6  C_4^{-1}(\varrho+1)^{-\ga_1+\ga} \ge 1$. Then proceeding similar to case (i) above, we obtain 
   \[ v_\e \leq \Ga \ov w_\rho \quad\mbox{in }\Om, \ \ \mbox{for all } \e\in(0,1),\]
  where $v_\e\in W^{s_1,p}_0(\Om)$ is the solution to problem \eqref{probsingeps}. Passing to the limit $\e\to 0$ in the above expression, we get 
  \begin{align*}
    v(x) \leq \Ga d(x)^{\frac{ps_1-\ga_1}{p-1+\de}} \quad\mbox{in }\Om,
  \end{align*}
 for all $\ga_1 \in (\ga,ps_1)$.\\
  \textit{Case (iii)}:  If $\ga-s_1(1-\de)\leq 0$.\\
 We fix $\tilde{\ga}>0$ and $\varsigma>0$ such that $\tilde{\ga}=(1-\de)s_1+\varsigma(p-1+\de)>s_1(1-\de)\geq \ga$. We further impose the conditions $\varsigma<\frac{p-1+p\de s_1}{p(p-1+\de)}$ (this implies $\tl\ga<1+s_1-\frac{1}{p}$) and $\varsigma\neq s_1-q's_2$ (this implies $\tl\ga\neq ps_1-q's_2(p-1+\de)$). Thanks to \eqref{kermain}, we can choose  $m\geq 1$ such that $K_{\ga}\leq m K_{\tilde{\ga}}$ in $\Om$. Let $\tilde{v}_{\tilde{\ga}}\in W^{s_1,p}_{\rm loc}(\Om)$ be the minimal solution to the problem:
  \begin{equation*}
 \left\{\begin{array}{rllll}
 (-\Delta)^{s_1}_{p}u+(-\Delta)^{s_2}_{q}u & = m K_{\tilde{\ga}}(x)u^{-\de}, \; \ u>0 \; \text{ in } \Om, \\
  \quad u&=0 \quad \text{ in } \mathbb{R}^N\setminus \Om.
 \end{array}
 \right.
 \end{equation*}
 By the choice of $m$, we see that $\tilde{v}$ is a super solution to problem \eqref{probsing}. Since $\tl\ga<1+s_1-1/p$, by applying the weak comparison principle as in proposition \ref{prop4}, and case (i) above, we get 
  \begin{align*}
  v(x) \leq \tilde{v}_{\tilde{\ga}}(x) \leq \Ga d(x)^\frac{s_1p-\tilde{\ga}}{p-1+\de}= \Ga d(x)^{s_1-\varsigma} \quad\mbox{in }\Om.
 \end{align*}
 Therefore, for any $\sg\in (0,s_1)$, we can choose $\varsigma\in (0,s_1-\sg)$ satisfying all the assumptions above. This completes the proof of the theorem.
  \QED
  \end{proof}
\textbf{Proof of Theorem \ref{bdryregsing}}: On account of remark \ref{remequivsoln} and theorem \ref{impintreg}, we have $v\in C^{0,\sg}_{\rm loc}(\Om)$, for all $\sg\in (0,s_1)$. This coupled with the boundary behavior of $v$, as in theorem \ref{uppdist}, proves the theorem. \QED
\textbf{Proof of Corollary \ref{corcriticalsing}}:  We first see that $u\in L^\infty(\Om)$. Indeed, for $m>0$, $t\ge 0$ and $\kappa\ge 0$, we define $g_\ka(t)=(t-1)_+ (\llcorner (t-1)_+ \lrcorner_m)^\ka$, where $\llcorner t \lrcorner_m:=\min\{t,m\}$. Set $G_\ka(t)=\int_{0}^{t}g^\prime_\ka(\tau)^{1/p}d\tau$. Then, proceeding similar to the proof of theorem \ref{mainbound}, we can prove that 
 \begin{align*}
 	\left(\int_{\Om} (u-1)^{p^*_{s_1}}_+ ([(u-1)_+]_m)^{\ka \frac{p^*_{s_1}}{p}}\right)^\frac{p}{p^*_{s_1}} &\leq C\| G_\ka(u)\|^p_{W^{s_1,p}_0(\Om)} \\
 	&\leq C \int_{\Om\cap\{u\geq 1\}} \big(1+u^{p^*_{s_1}-1}\big)(u-1)(\llcorner (u-1)_+ \lrcorner_m)^\ka.
 \end{align*}
 Thus, we can conclude that $(u-1)_+\in L^\infty(\Om)$, and consequently $u\in L^\infty(\Om)$. Therefore,
\begin{align*}
	(-\Delta)^{s_1}_{p}u+  (-\Delta)^{s_2}_{q}u = \la u^{-\de}+ b(x,u) \leq C_b \big(\la + \|u\|^{\de}_{L^\infty(\Om)} +\|u\|^{p^*_{s_1}-1+\de}_{L^\infty(\Om)}\big) u^{-\de}:= \la_* u^{-\de}.
\end{align*}
Let $v\in W^{s_1,p}_0(\Om)$ be the minimal solution of $(\mc S_{0,\de})$ with $K_0(x):=\la_*$ in there, then by applying theorem \ref{uppdist} for the case $\ga=0$ and the weak comparison theorem,  we obtain
\begin{align*}
	u(x)\leq v(x)\leq \Ga d(x)^{s_1-\sg} \quad\mbox{in }\Om,
\end{align*}
for all $\sg\in (0,s_1)$ and for some $\Ga>0$, depending on $\la$, $p,\de,r$ and $\|u\|_{L^\infty(\Om)}$ only. By combining the result of Theorem \ref{impintreg} with the above boundary behavior of $u$, we can prove that  $u\in C^{0,s_1-\sg}(\ov\Om)$, for all $\sg\in (0,s_1)$ (the proof is similar to \cite[Theorem 1.1]{iann}).\QED

Now, we prove the following strong comparison principle for singular problems, essentially due to Jarohs \cite[Theorem 1.1]{jarohs} for non-singular problems.\\
\textbf{Proof of Theorem \ref{strongcomp}}: 
Let $D\subset\mb R^N$ such that $0<|D|<|\Om|$ and $l:=\mathrm{inf}_D(v-w)>0$ (if we can not find such a $D$, then we have $v\equiv w$). Without loss of generality we may assume $l\leq 1$. Fix $K\Subset \Om\setminus D$ and $f\in C^2_c(\Om\setminus D)$ such that $f\equiv 1$ in $K$ and $0\le f \leq 1$. Let $a_0,b>0$ be given by \cite[Lemma 3.6]{jarohs} and $a\in (0,a_0]$ with $a<\eta \mathrm{dist}(\mathrm{supp}(f),\Om)^{s_1}$. Define $u_a= v-af-l\chi_D$, then $u_a\in W^{s_1,p}_0(\mathrm{supp}f)$. Therefore, using \cite[Lemma 3.6]{jarohs} for $(-\De)_p^{s_1}$ and $(-\De)_q^{s_2}$ (thanks to the assumption on $\al$), we obtain in the weak sense in supp$f$,
 \begin{align}\label{eq28}
 	(-\Delta)^{s_1}_{p}u_a+  (-\Delta)^{s_2}_{q}u_a \ge (-\Delta)^{s_1}_{p}v+  (-\Delta)^{s_2}_{q}v +b &\geq v^{-\de}+g+b  \nonumber\\
 	&= u_a^{-\de}+g+b+(v^{-\de}-u_a^{-\de}) \nonumber\\
 	&\ge  u_a^{-\de}+g+b+(v^{-\de}-(v-a)^{-\de}).
 \end{align}
  We note that the terms involving power ${-\de}$ are well defined by the relation $v\geq \eta d^{s_1}$ and the choice of $a$. Furthermore, since $\de\in(0,1)$, we have 
 \begin{align*}
 	(v-a)^{-\de}-v^{-\de}\leq a^{\de}v^{\de}(v-a)^{\de} \leq a^\de v^{2\de}\leq a^\de \|v\|^{2\de}_{L^\infty(\Om)}.
 \end{align*}
 Therefore, we can choose $a>0$ sufficiently small such that $b-(v^{-\de}+(v-a)^{-\de})>0$ in supp$f$. Thus, from \eqref{eq28}, we get
 \begin{align*}
  (-\Delta)^{s_1}_{p}u_a+  (-\Delta)^{s_2}_{q}u_a \geq u_a^{-\de}+g \quad\mbox{in }\mathrm{supp}f,
 \end{align*}	
 and $u_a\leq w$ in $\mb R^N\setminus\mathrm{supp}f$. Consequently, using the weak comparison principle, we get $u_a \geq w$ in supp$f$. In particular, 
 \begin{align*}
  v \geq w + a \quad\mbox{in }K.
 \end{align*}
 Since $K$ and $f$ were chosen arbitrarily, we have 
 \[ \mathrm{inf}_K (v-w)>0 \quad\mbox{for all }K\Subset\Om\setminus D. \]
 Now, repeating the process for $\tilde{D}\Subset \Om\setminus D$ with $\tilde{l}:=\mathrm{inf}_{\tilde{D}} (v-w)>0$, we complete the proof of the theorem for the case $v\in C^{0,\al}_{\mathrm{loc}}(\Om)$. The other case can be dealt in a similar manner by following the approach of \cite[Theorem 1.1]{jarohs}.\QED


\end{document}